\documentclass[10pt,a4paper]{amsart}
\pdfoutput=1

\usepackage{lmodern}                    
\usepackage[T1]{fontenc}
\usepackage[utf8]{inputenc}

\usepackage[british]{babel}

\usepackage{amsmath,amssymb,amsthm,amsfonts,amsbsy,latexsym}
\usepackage[abbrev,alphabetic]{amsrefs}
\usepackage{mathtools}
\usepackage{commath}

\usepackage{url}
\usepackage[breaklinks,linktocpage]{hyperref}
\usepackage[capitalise]{cleveref}

\newcommand*{\doi}[1]{\href{https://doi.org/#1}{doi:#1}}
\usepackage{xifthen}
\newcommand*{\msc}[2][]{\href{https://mathscinet.ams.org/mathscinet/search/mscdoc.html?code=#2,(#1)}{Primary: #2\ifthenelse{\isempty{#1}}{}{; Secondary: #1}.}}

\usepackage{tikz}
\usetikzlibrary{arrows}
\usetikzlibrary{babel}
\usetikzlibrary{cd}
\usetikzlibrary{decorations.pathreplacing}
\usetikzlibrary{matrix}
\usetikzlibrary{quotes}

\makeatletter
\protected\def\tikz@nonactivecolon{\ifmmode\mathrel{\mathop\ordinarycolon}\else:\fi}
\makeatother

\usepackage{todonotes} 
\usepackage{enumitem}
\usepackage{csquotes}
\usepackage{verbatim}

\usepackage{etoolbox}
\apptocmd{\sloppy}{\hbadness 10000\relax}{}{}
\usepackage{microtype}
\usepackage{scrhack}

\theoremstyle{definition}
\newtheorem{defi}{Definition}[section]
\crefname{defi}{Definition}{Definitions}

\theoremstyle{plain}
\newtheorem{lemm}[defi]{Lemma}
\crefname{lemm}{Lemma}{Lemmata}
\newtheorem{theo}[defi]{Theorem}
\crefname{theo}{Theorem}{Theorems}

\newtheorem{prop}[defi]{Proposition}

\newtheorem*{clai}{Claim}
\crefname{theoenum}{Theorem}{Theorems}

\theoremstyle{remark}
\newtheorem{rema}[defi]{Remark}
\crefname{rema}{Remark}{Remarks}
\newtheorem{exam}[defi]{Example}

\newtheoremstyle{maintheorem}{}{}{\itshape}{}{\bfseries}{}{.5em}{#1 \!\thmnote{#3}.}
\theoremstyle{maintheorem}
\newtheorem*{mainthm}{Theorem}

\newcommand*{\itemtitle}[1]{\textbf{\emph{(#1). }}}

\newcommand{\overbar}[1]{\mkern 1.5mu\overline{\mkern-1.5mu#1\mkern-1.5mu}\mkern
1.5mu}
\newcommand*{\NN}{\mathbb{N}}
\newcommand*{\ZZ}{\mathbb{Z}}
\newcommand*{\QQ}{\mathbb{Q}}
\newcommand*{\RR}{\mathbb{R}}

\let\leq\leqslant
\let\geq\geqslant
\let\phi\varphi

\DeclareMathOperator{\id}{id}

\DeclareMathOperator{\cone}{cone}
\DeclareMathOperator{\supp}{supp}
\DeclareMathOperator{\pr}{pr}

\DeclareMathOperator{\Aut}{Aut}
\DeclareMathOperator{\res}{res}

\DeclareMathOperator{\GL}{GL}
\newcommand*{\covering}[1]{\overbar{#1}}

\DeclareMathOperator{\Ore}{Ore}

\newcommand*{\Whw}{\ensuremath{\operatorname{Wh}^\omega}}
\newcommand*{\Kw}{\ensuremath{K_1^\omega}}
\newcommand*{\Kwr}{\ensuremath{{\widetilde K}_1^\omega}}
\newcommand*{\Kr}{\ensuremath{{\widetilde K}_1}}
\newcommand*{\torsionu}{\ensuremath{\rho_u^{(2)}}}
\newcommand*{\Poly}{{\ensuremath{\mathcal{P}}}}
\newcommand*{\PolyT}{{\ensuremath{\mathcal{P}_T}}}
\newcommand*{\bl}{\ensuremath{b^{(2)}}}

\newcommand*{\Dab}[1][D]{{#1}^\times_{\textrm{ab}}}
\newcommand*{\bag}[1][D]{b^{#1}}

\newcommand*{\chiag}[1][D]{\chi^{#1}}
\newcommand*{\Pag}[1]{P^{#1}}
\newcommand*{\PLtwo}{P_{L^2}}
\newcommand*{\torsionag}[1][D]{\rho_{#1}}
\newcommand*{\linnelld}{\mathcal{D}}

\begin{document}

\title{Agrarian and $L^2$-Invariants}

\author{Fabian Henneke}
\email{\href{mailto:henneke@uni-bonn.de}{henneke@uni-bonn.de}}
\address{Max-Planck-Institut für Mathematik, Vivatsgasse 7, 53111 Bonn, Germany}

\author{Dawid Kielak}
\email{\href{mailto:dkielak@math.uni-bielefeld.de}{dkielak@math.uni-bielefeld.de}}
\address{Fakultät für Mathematik, Universität Bielefeld, Postfach 100131, 33501 Bielefeld, Germany}

\subjclass[2010]{\msc[12E15, 16S35, 20E06, 57Q10]{20J05}}

\begin{abstract}
We develop the theory of agrarian invariants, which are algebraic counterparts to $L^2$-invariants. Specifically, we introduce the notions of agrarian Betti numbers, agrarian acyclicity, agrarian torsion and agrarian polytope for finite free $G$-CW complexes together with a fixed choice of a ring homomorphism from the group ring $\ZZ G$ to a skew field. For the particular choice of the Linnell skew field $\linnelld(G)$, this approach recovers most of the information encoded in the corresponding $L^2$-invariants.

As an application, we prove that for agrarian groups of deficiency $1$, the agrarian polytope admits a marking of its vertices which controls the Bieri--Neumann--Strebel invariant of the group, improving a result of the second author and partially answering a question of Friedl--Tillmann.

We also use the technology developed here to prove the Friedl--Tillmann conjecture on polytopes for two-generator one-relator groups; the proof forms the contents of another article.
\end{abstract}

\maketitle


\section{Introduction}

In 2017, Friedl--Lück~\cite{FL2017} broadened the arsenal of $L^2$-invariants: to the more classical $L^2$-Betti numbers and $L^2$-torsion, they added the \emph{universal $L^2$-torsion} and the \emph{$L^2$-torsion polytope}.
The definitions involve some operator algebra technology, and a little $K$-theory. Practically however, it transpired, e.g., in~\cites{FunkeKielak2018,Kielak2018}, that these new $L^2$-invariants can be computed using linear algebra over skew fields. The current article redevelops much of the $L^2$-theory from this (more algebraic) point of view. The resulting theory of \emph{agrarian invariants} is at the same time conceptually simpler and applicable to a greater number of  groups.

\subsubsection*{The invariants}
Recall that a group $G$ is \emph{agrarian} if its integral group ring $\ZZ G$ embeds in a skew field. This terminology was introduced in~\cite{Kielak2018}, but the idea dates back to Malcev~\cite{Malcev1948}, and is a central theme of the work of Cohn~\cite{Cohn1995}.
Taking a specific \emph{agrarian embedding} $\ZZ G \hookrightarrow D$ for some skew field $D$, or more generally an \emph{agrarian map} $\ZZ G \to D$, allows us to define the notion of $D$-agrarian Betti numbers: when $G$ acts cellularly on a CW-complex $X$, we simply compute the $D$-dimension of the homology of $D \otimes_{\ZZ G} C_\ast$, where $C_\ast$ is the cellular chain complex of $X$.
When $D$ is the skew field $\mathcal D(G)$ introduced by Linnell in~\cite{Linnell1993} (assuming that $G$ is torsion-free and satisfies the Atiyah conjecture), the $D(G)$-agrarian Betti numbers are precisely the $L^2$-Betti numbers.
We show in \cref{prop:betti:independence} that for two non-equivalent agrarian embeddings, there is always a CW complex whose agrarian Betti numbers with respect to the two embeddings differ.

When the agrarian Betti numbers vanish and $G$ acts on $X$ cocompactly, we define the \emph{agrarian torsion}, in essentially the same way as Whitehead or Reidemeister torsion is defined. Again, when $\mathcal D(G)$ is Linnell's skew field, we obtain an invariant very closely related to the universal $L^2$-torsion. In fact, in this case agrarian and universal $L^2$-torsion often contain the same amount of information by a theorem of Linnell--Lück~\cite{LL2018}.

The vanishing of $L^2$-Betti numbers is guaranteed when $X$ fibres over the circle due to a celebrated theorem of Lück; the agrarian Betti numbers also vanish in this setting, as we show in \cref{theo:betti:torus}, provided that the agrarian map used satisfies the additional technical condition of being \emph{rational} (see \cref{defi:agrarian:rational}). Let us remark here that every agrarian map can be turned into a rational one, whose target we will usually denote by $D_r$.

The final invariant, the \emph{agrarian polytope}, is a little more involved. In the context of $L^2$-invariants, one can write the universal $L^2$-torsion as a fraction of two elements of a (twisted) group ring of the free part of the abelianisation of $G$. Both the numerator and the denominator can be converted into polytopes, using the Newton polytope construction, and the $L^2$-torsion polytope is defined as the formal difference of these Newton polytopes. The $L^2$-torsion polytope naturally lives in the polytope group of $G$, defined in~\cite{FL2017} and investigated further by Funke~\cite{Funke2016}.

In the agrarian setting it is precisely the notion of rationality which allows us to express the agrarian torsion as a fraction of two elements of a (twisted) group ring of the free part of the abelianisation of $G$, in complete analogy to the $L^2$ case.
The agrarian polytope is then constructed in the same way as the $L^2$-torsion polytope.

\smallskip
The main advantage of agrarian invariants over $L^2$-invariants lies in the fact that they are defined for a group $G$ as long as $\ZZ G$ maps to \emph{any} skew field -- not necessarily the one known to exist if $G$ were to satisfy the Atiyah conjecture. Even when we require the agrarian map to be injective, the class of agrarian groups is \emph{a priori} larger than the class of torsion-free groups satisfying the Atiyah conjecture. Moreover, there are explicit classes
of groups of topological interest, such as surface-by-surface groups, which are known to be agrarian, but at the time of writing were not known to satisfy the Atiyah conjecture (see~\cite{Kielak2018}*{Section~4} for a more substantial discussion).

An inconvenience that comes with our more general approach is that for a torsion-free group $G$ not known to satisfy the Atiyah conjecture, there is no longer a canonical choice of an agrarian embedding of $G$. In general, different agrarian embeddings will lead to differing values for the associated agrarian invariants, which makes it important to keep track of the embedding used to define them.

\subsubsection*{Applications}
We introduce the theory of agrarian invariants not only for its general appeal, but also with specific applications in mind.

For our first application, we turn our attention to agrarian groups of deficiency $1$. These groups were already investigated in~\cite[Section~5.6]{Kielak2018}, where it was shown that for such a group $G$ there exists a marked polytope determining the BNS invariant $\Sigma^1$ of $G$, which is a subset of $H^1(G;\RR)\setminus\{0\}$ and can be defined for general finitely generated groups.
Here we prove that the polytope can be replaced by an agrarian polytope endowed with a marking of a much more rigid character.
This result constitutes progress towards answering~\cite[Question~9.4]{FT2015} in the following way:

\begin{mainthm}[\ref{theo:deficency:marked}]
Let $G$ be a $D$-agrarian group of deficiency 1. There exists a marking of the vertices of the agrarian polytope $\Pag{D_r}(G)$ such that for every $\phi \in H^1(G;\RR) \smallsetminus \{0\}$ we have $\phi \in \Sigma^1(G)$ if and only if $\phi$ is marked.
\end{mainthm}
We refer to \cref{marked def} for details on how the marking on the polytope yields a marking of classes in the first cohomology.

\smallskip

In \cite{FT2015}, Friedl--Tillmann assigned a marked polytope to a fixed presentation of a torsion-free two-generator one-relator group.
By recognizing this polytope as an agrarian polytope, we are able to prove that their construction does in fact not depend on the chosen presentation.
We can also relate the thickness of the polytope in a given direction to another agrarian invariant, which in the case of two-generator one-relator groups will turn out to compute a measure of complexity for possible HNN splittings of the group. The proofs of these two facts are the contents of \cite{HennekeKielak2019a} by the authors.

\smallskip
Let us also mention another connection of interest: in the setting of $3$-manifolds, the $L^2$-torsion polytope was shown in~\cite{FL2017} to coincide with the Thurston polytope, that is the dual of the unit ball of the Thurston norm. This point of view allowed Funke and the second author to define the Thurston norm in the setting of free-by-cyclic groups~\cite{FunkeKielak2018}. Since the $L^2$-torsion polytope is a particular example of an agrarian polytope, one can expect that the theory of agrarian invariants will be useful in defining further generalisations of the Thurston norm.

The Thurston polytope is also closely related to the BNS invariants via~\cite[Theorem~5]{Thurston1986} and~\cite[Corollary~F]{Bierietal1987}. A new proof of this relation, essentially using agrarian methods, was given by the second author in~\cite{Kielak2018}.

\subsection*{Acknowledgements}
The first author would like to thank his advisor Wolfgang Lück as well as Stefan Friedl and Xiaolei Wu for helpful discussions. He is especially grateful to Stefan Friedl for an invitation to Regensburg and the hospitality experienced there.

The present work is part of the first author's PhD project at the University of Bonn. He was supported by Wolfgang Lück's ERC Advanced Grant ``KL2MG-interactions'' (no. 662400) granted by the European Research Council. The second author was supported by the grant KI 1853/3-1 within the Priority Programme 2026 \href{https://www.spp2026.de/}{`Geometry at Infinity'} of the German Science Foundation (DFG).

\section{Agrarian groups and Ore fields}

In the following, all rings are associative, unital, and not necessarily commutative; ring homomorphisms preserve the unit.
Modules are understood to be left modules unless specified otherwise. As a general exception, we consider a ring $R$ to be an $R$-$R$-bimodule.

Let $G$ be a group and denote by $\ZZ G$ the integral group ring of $G$.
\begin{defi}
    Let $G$ be a group.
    A ring homomorphism $\alpha\colon \ZZ G\to D$ to a skew field $D$ is called an \emph{agrarian map for} $G$.
    A morphism between two agrarian maps is an inclusion of skew fields that together with the maps from $\ZZ G$ forms a commutative triangle.
\end{defi}

While many formal properties of the agrarian Betti numbers we will introduce below hold in the situation of an arbitrary agrarian map, concrete calculations and definitions of higher invariants usually require the map to be injective:

\begin{defi}
    Let $G$ be a group.
    An \emph{agrarian embedding for} $G$ is an injective agrarian map.
    If $G$ admits an agrarian embedding (into a skew field $D$), it is called a \emph{($D$-)agrarian group}.
\end{defi}

An agrarian group $G$ is necessarily torsion-free; also, it satisfies the Kaplansky zero divisor conjecture, that is, $\ZZ G$ has no non-trivial zero divisors.

At present, there are no known torsion-free examples of groups which are not agrarian.
There is however a plethora of positive examples of agrarian groups:
If a torsion-free group satisfies the Atiyah conjecture over $\QQ$, then its group ring embeds into the skew field $\linnelld(G)=\linnelld(G;\QQ)$ of Linnell (by the same argument as used in~\cite{Luck2002}*{Theorem~10.39}), and hence the group is agrarian. Similarly, any biorderable group is agrarian, since its integral group ring embeds into the Malcev--Neumann completion, which is a skew field by~\cites{Malcev1948,Neumann1949} (see also~\cite{Cohn1995}*{Theorem~2.4.5}).

Agrarian groups also enjoy a number of convenient inheritance properties. For a more substantial discussion, see~\cite{Kielak2018}*{Section 4} by the second author.
One example of agrarian groups which does not feature in the discussion in~\cite{Kielak2018} is that of torsion-free one-relator groups. All such groups have been shown to be agrarian by Lewin--Lewin~\cite{LL1978}.

\subsection{Twisted group rings} We now recall the construction of twisted group rings, which will later allow us to obtain from an arbitrary agrarian map an agrarian map into a localised twisted polynomial ring.

\begin{defi}
Let $R$ be a ring and let $G$ be a group. Let functions $c\colon G\to\Aut(R)$ and $\tau \colon G\times G\to R^\times$ be such that
\begin{align*}
    c(g)\circ c(g') &= c_{\tau(g, g')} \circ c(gg')\\
    \tau(g,g')\tau(gg', g'') &= c(g)(\tau(g', g''))\tau(g, g'g''),
\end{align*}
where $g,g', g''\in G$, and where $c_r\in \Aut(R)$ for $r\in R^\times$ denotes the conjugation map $x\mapsto rxr^{-1}$. The functions $c$ and $\tau$ are called \emph{structure functions}. We denote by $RG$ the free $R$-module with basis $G$ and write elements of $RG$ as finite $R$-linear combinations $\sum_{g\in G} \lambda_g\ast g$ of elements of $G$. When convenient, we shorten $1\ast g$ to $g$. The structure functions endow $RG$ with the structure of an (associative) \emph{twisted group ring} by declaring
\[
g \cdot (r \ast 1) = c(g)(r) \ast g \textrm{ and }
g \cdot g'= \tau(g,g') \ast gg'
\]
and extending linearly.
\end{defi}

The usual, \emph{untwisted} group ring corresponds to the special case of trivial structure functions. In the following, group rings with $R=\ZZ$ will always be understood to be untwisted.

The fundamental example of a twisted group ring arises in the following way:
\begin{exam}
\label{exam:agrarian:twisted}
Let $\phi\colon G\twoheadrightarrow H$ be a group epimorphism with kernel $K$. We choose any set-theoretic section $s\colon H\to G$, i.e., a map between the underlying sets such that $\phi \circ s = \id_H$. We denote by $(\ZZ K)H$ the twisted group ring defined by the structure functions $c(h)(r)=s(h)rs(h)^{-1}$ and $\tau(h, h')=s(h)s(h')s(hh')^{-1}$. The untwisted group ring $\ZZ G$ is then isomorphic to the twisted group ring $(\ZZ K)H$ via the map
\[g\mapsto \left(g \cdot (s\circ\phi)(g)^{-1}\right)\cdot\phi(g).\qedhere\]
\end{exam}

Observe that the structure maps of the twisted group ring $(\ZZ K) H$ we just defined depend on the section $s \colon H \to G$. The global ring structure however does not, since for any section the resulting twisted group ring is isomorphic (as a ring) to $\ZZ G$.

The construction of a twisted group ring out of a group epimorphism can be extended to agrarian maps. This technique is formulated in the following technical lemma, which will be our main source of twisted group rings:

\begin{lemm}
\label{lemm:agrarian:equivariant_embedding}
Let $\alpha\colon \ZZ G\to D$ be an agrarian map for a group $G$. Let $N\leq G$ be a normal subgroup and set $Q\coloneqq G/N$. Then $\alpha$ restricts to an agrarian map $\ZZ N \to D$ for $N$ that is equivariant with respect to the conjugation action of $G$. Moreover, for any section $s \colon Q \to G$ of the quotient map, $\alpha$ extends to a ring homomorphism
\[ (\ZZ N)Q\to DQ,\]
where $(\ZZ N)Q$ is as defined in \cref{exam:agrarian:twisted}, and $DQ$ is a twisted group ring with the same structure functions as $(\ZZ N)Q$.
The twisted group ring structure of $D Q$ is independent of the choice of the section $s$ up to isomorphism, and this isomorphism can be chosen to map $\sum_{q\in Q} u_q q$ to $\sum_{q\in Q} v_q q$ such that for every $q \in Q$, the elements $u_q$ and $v_q$ differ only by an element of $D^\times$.
\end{lemm}
\begin{proof}
By definition, $\alpha$ restricts to an agrarian map for $N$.
Note that an element $g\in G$ acts on $D$ by conjugation with $\alpha(g)$, which is always invertible in $D$ since $g$ is invertible in $\ZZ G$.
Since $N$ is normal in $G$, the conjugation action of $G$ on $\ZZ G$ preserves $\ZZ N$ and hence induces an action on $\ZZ N$.
The restricted agrarian map $\alpha\colon\ZZ N\to D$ is equivariant with respect to these actions by construction.

Let $s\colon Q\to G$ be a set-theoretic section of the group epimorphism $\pr\colon G\to G/N = Q$, i.e., a map that satisfies $\pr\circ s=\id_Q$. Then every $q\in Q$ defines an automorphism $c_{s(q)}$ of $\ZZ N$ given by conjugation by $s(q)$. We denote by $(\ZZ N)Q$ the twisted group ring associated to the section $s$ as in~\cref{exam:agrarian:twisted}. Since the automorphism $c_{s(q)}$ of $\ZZ N$ extends to an automorphism of $D\supset \ZZ N$, which is given by conjugation by $c_{\alpha(s(q))}$, we can apply the twisted group ring construction also to $D$ and $Q$ in such a way that the map $(\ZZ N)Q\hookrightarrow DQ$ extends $\ZZ N\hookrightarrow D$.

\smallskip

Let $s_1$ and $s_2$ be two set-theoretic sections of $\pr\colon G\to Q$. Denote by $D_{s_1}Q$ and $D_{s_2}Q$ the associated twisted group ring structures on $DQ$. We claim that the map $\Phi\colon D_{s_1}Q\to D_{s_2}Q$ given by
\[ \sum_{q\in Q} u_q \ast q\mapsto \sum_{q\in Q} (u_q s_1(q) s_2(q)^{-1}) \ast q\]
is a ring isomorphism. Since $s_1(q) s_2(q)^{-1}\in G\subset D^\times$ for all $q\in Q$, it is clear that $\Phi$ is an isomorphism between the underlying free $D$-modules and changes the coefficients by a unit in $D$ only.

It thus remains to check that $\Phi$ respects the ring multiplications, which is fully determined by the cases $q\cdot u$ and $q\cdot q'$ for $q,q'\in Q$ and $u\in D$. We denote by $\tau_i$ for $i=1,2$ the respective structure function of $D_{s_i}Q$ and compute that

\begin{align*}
\Phi(q\cdot q') &= \Phi(\tau_1(q, q')\ast qq') \\
&= \Phi(s_1(q)s_1(q')s_1(qq')^{-1} \ast qq')\\
&= s_1(q)s_1(q')s_1(qq')^{-1}s_1(qq')s_2(qq')^{-1} \ast qq'\\
&= s_1(q)s_1(q')s_2(qq')^{-1} \ast qq'
\end{align*}
\begin{align*}
\Phi(q)\cdot \Phi(q') &= (s_1(q)s_2(q)^{-1}\ast q)\cdot (s_1(q')s_2(q')^{-1}\ast q') \\
&= (s_1(q)s_2(q)^{-1} s_2(q)s_1(q')s_2(q')^{-1}s_2(q)^{-1} \ast q)\cdot q'\\
&= s_1(q)s_1(q')s_2(q')^{-1}s_2(q)^{-1}s_2(q)s_2(q')s_2(qq')^{-1}\ast qq'\\
&= s_1(q)s_1(q')s_2(qq')^{-1} \ast qq'
\end{align*}
and
\begin{align*}
\Phi(q\cdot u\ast 1) &= \Phi(s_1(q) u s_1(q)^{-1} \ast q \cdot 1\ast 1)\\
&= \Phi(s_1(q) u s_1(q)^{-1} s_1(q) s_1(1) s_1(q)^{-1} \ast q)\\
&= \Phi(s_1(q) u s_1(1) s_1(q)^{-1} \ast q)\\
&= s_1(q) u s_1(1) s_1(q)^{-1} s_1(q) s_2(q)^{-1} \ast q\\
&= s_1(q) u s_1(1) s_2(q)^{-1} \ast q
\end{align*}
\begin{align*}
\Phi(q)\cdot \Phi(u\ast 1) &= (s_1(q)s_2(q)^{-1}\ast q)\cdot (u s_1(1)s_2(1)^{-1}\ast 1)\\
&= (s_1(q)s_2(q)^{-1}s_2(q)u s_1(1)s_2(1)^{-1}s_2(q)^{-1} \ast q) \cdot 1\\
&= s_1(q) u s_1(1)s_2(1)^{-1}s_2(q)^{-1} s_2(q)s_2(1)s_2(q)^{-1} \ast q\\
&= s_1(q) u s_1(1)s_2(q)^{-1} \ast q.
\end{align*}
Hence $\Phi$ is an isomorphism of rings.
\end{proof}

\subsection{The rationalisation of an agrarian map}
We now consider the case of an agrarian map $\alpha\colon\ZZ G\to D$ and a normal subgroup $K\leq G$ such that $G/K$ is a finitely generated free abelian group $H$. \Cref{lemm:agrarian:equivariant_embedding} then provides us with a twisted group ring $DH$ and a ring homomorphism $\ZZ G \cong (\ZZ K) H \to DH$. Since $H$ is free abelian, it is in particular biorderable and hence $DH$ contains no non-trivial zero divisors (this is a standard fact following from the existence of an embedding of $D H$ into its Malcev--Neumann completion; for details see
~\cite[Theorem~4.10]{Kielak2018}). It then follows from~\cite[Theorem~2.11]{Kielak2018}, an extension of a result of Tamari~\cite{Tamari1954} to twisted skew field coefficients, that $DH$ satisfies the Ore condition, and thus its Ore localisation is a skew field.

Recall that a ring $R$ without non-trivial zero divisors satisfies the \emph{Ore condition} if for every $p,q \in R$ with $q \neq 0$ there exists $r,s \in R$ with $s \neq 0$ such that
\[
ps = qr.
\]
This equality allows for the conversion of a left fraction $ q^{-1} p$ into a right fraction $r s^{-1}$, which in turn makes it possible to multiply fractions (in the obvious way). The Ore condition also facilitates the existence of common denominators, and thus allows for addition of fractions. Thanks to these properties, the ring $R$ embeds into its \emph{Ore localisation} \[\Ore(R)\coloneqq\{ q^{-1} p \mid p,q \in R, q \neq 0 \},\]
which is evidently a skew field.
For details see the book of Passman~\cite[Section 4.4]{Passman1985}.

Our construction is summarised in
\begin{defi}
\label{defi:agrarian:ore_field}
Let $\alpha\colon \ZZ G\to D$ be an agrarian map for a group $G$. Let $K$ be a normal subgroup of $G$ such that $H\coloneqq G/K$ is finitely generated free abelian. The \emph{$K$-rationalisation} of $\alpha$ is the composition
\[\ZZ G\cong (\ZZ K)H\to DH \hookrightarrow \Ore(DH),\]
where $\Ore(DH)$ is the Ore field of fractions of the twisted group ring $DH$ of \cref{lemm:agrarian:equivariant_embedding}.
\end{defi}

The construction of the $K$-rationalisation of an agrarian map $\ZZ G\to D$ of course depends on a choice of a set-theoretic section of the projection $G\to G/K$, which we will assume to be fixed once and for all for any group $G$ being considered.
By \cref{lemm:agrarian:equivariant_embedding}, at least the target skew field of the $K$-rationalisation is independent of this choice up to isomorphism.

Also note that the $K$-rationalisation of an agrarian embedding is again an embedding.

The typical situation in which we consider the $K$-rationalisation of a given agrarian map is that where $K$ is the kernel of the projection of $G$ onto the free part of its abelianisation.
\begin{defi}
\label{defi:agrarian:rational}
Let $G$ be a finitely generated group and let $\alpha\colon \ZZ G\to D$ be an agrarian map for $G$.
Denote the free part of the abelianisation of $G$ by $H$ and the kernel of the canonical projection of $G$ onto $H$ by $K$.
For this particular choice of $K$, we simply call the $K$-rationalisation of $\alpha$ the \emph{rationalisation}.

An agrarian map $\beta \colon \ZZ G \to E$ will be called \emph{rational} if there exists a skew field embedding $\iota \colon E \to F$ such that $\iota \circ \beta$ is the rationalisation of $\beta$.
\end{defi}
The term \emph{rational} is chosen to indicate that the target of a rational agrarian map should be viewed as a skew field of rational functions in finitely many variables with coefficients in a skew field.
While the special structure of rational functions is crucial for the development of the theory of agrarian invariants, the specific choice of the skew field of coefficients is mostly immaterial.

\section{Agrarian Betti numbers}
From now on, $G$ will denote a group with a fixed agrarian map $\alpha\colon\ZZ G\to D$. Unless indicated explicitly, all tensor products will be taken over $\ZZ G$.

\subsection{Definition of agrarian Betti numbers}
Let $C_*$ be a $\ZZ G$-chain complex and let $n\in \ZZ$. Viewing $D$ as a $D$-$\ZZ G$-bimodule via the agrarian map $\alpha$, the chain complex $D\otimes C_*$ becomes a $D$-chain complex. Since $D$ is a skew field, the $D$-module $H_p(D\otimes C_*)$ is free and we can consider its dimension $\dim_D H_p(D\otimes C_*)\in \NN \sqcup \{\infty\}$. This leads to the following definition of agrarian Betti numbers:

\begin{defi}
Let $G$ be a group with an agrarian map $\alpha\colon\ZZ G\to D$ and $C_*$ a $\ZZ G$-chain complex. For $n\in \ZZ$, the \emph{$n$-th $D$-Betti number} of $C_*$ with respect to $\alpha$ is defined as
\[\bag_n(C_*)\coloneqq \dim_D H_n(D\otimes C_*)\in \NN \sqcup \{\infty\},\]
A $\ZZ G$-chain complex is called \emph{$D$-acyclic} if all of its $D$-Betti numbers are equal to $0$.
\end{defi}

If the agrarian map $\alpha$ is chosen to be the augmentation homomorphism
\[\ZZ G\to \ZZ \hookrightarrow \QQ,\]
the associated agrarian Betti numbers of a $\ZZ G$-chain complex $C_*$ reduce to the ordinary Betti numbers of the quotient complex $C_*/G$.
In the special case where $G$ satisfies the Atiyah conjecture and the agrarian map is chosen to be the agrarian embedding $\ZZ G\hookrightarrow \linnelld(G)$, the $D$-Betti numbers of any $\ZZ G$-chain complex $C_*$ agree with its $L^2$-Betti numbers $\bl_*(C_*;G)$ by~\cite[Theorem~3.6~(2)]{FL2016}. Note that the assumption that $C_*$ is projective is not used in the proof, the theorem thus holds for arbitrary $\ZZ G$-chain complexes.

We will mainly be concerned with the agrarian Betti numbers assigned to \emph{$G$-CW-complexes}, which are equivariant analogues of CW-complexes and very convenient models for $G$-spaces. A typical example of a $G$-CW-complex is the universal covering of a connected CW-complex $X$ with $G=\pi_1(X)$.

\begin{defi}
Let $G$ be a (discrete) group. A \emph{$G$-CW-complex} is a CW-complex $X$ together with an implicit (left) $G$-action mapping $p$-cells to $p$-cells and such that any cell mapped into itself is fixed pointwise by the action. An action satisfying these properties is called \emph{cellular}. The \emph{$p$-skeleton} of a $G$-CW-complex $X$, denoted by $X_p$, is the $p$-skeleton of the underlying CW-complex together with the restriction of the $G$-action. If $X=X_p$ for some $p$ and $p$ is minimal with this property, then $X$ is said to be \emph{of dimension $p$}. Any $G$-orbit of a $p$-dimensional cell of the underlying CW-complex constitutes a \emph{$p$-dimensional $G$-cell} of $X$. A $G$-CW-complex is \emph{connected} if the underlying CW-complex is connected. The \emph{cellular ($\ZZ G$-)chain complex} $C_*(X)$ is obtained from the cellular chain complex of the underlying CW-complex by considering the induced left action by $G$. All differentials are $\ZZ G$-linear.
\end{defi}

\begin{defi}
A $G$-CW-complex $X$ is called \emph{free} if its $G$-action is free. It is called \emph{of finite type} if for every $p\geq 0$ there are only finitely many $p$-dimensional $G$-cells in $X$. If the total number of $G$-cells of any dimension in $X$ is finite, the $G$-CW-complex $X$ is called \emph{finite}.
\end{defi}
\begin{defi}
A $\ZZ G$-chain complex $C_*$ is called \emph{free} (\emph{of finite type}) if $C_n$ is a free (finitely generated) $\ZZ G$-module for every $n\in\ZZ$. It is called \emph{bounded} if there is $N\in\NN$ such that $C_i=0$ for $i>N$ and $i<-N$.
\end{defi}

If a $G$-CW-complex is free (of finite type), then its associated cellular chain complex $C_*(X)$ is free (of finite type). If it is finite, then its cellular chain complex is bounded and of finite type.

We now define agrarian Betti numbers for $G$-CW-complexes:

\begin{defi}
Let $X$ be an $G$-CW-complex. For $p\geq 0$, the \emph{$p$-th $D$-Betti number} of $X$ with respect to $\alpha$ is defined as
\[\bag_p(X)\coloneqq \bag_p(C_*(X)) \in \NN \sqcup \{\infty\}.\]
A $G$-CW-complex $X$ is called \emph{$D$-acyclic} if all of its $D$-Betti numbers are equal to $0$.
\end{defi}

If $X$ is an $G$-CW-complex of finite type, then the $D$-Betti numbers of $X$ will all be non-negative integers.

\subsection{Dependence on the agrarian map}
Note that the agrarian Betti numbers may depend not only on the skew field $D$ but also on the particular choice of agrarian map.
We will first consider a typical situation in which two different agrarian maps define the same agrarian Betti numbers.

\begin{rema}
\label{rema:betti:epic}
Recall that a ring homomorphism $f\colon R\to S$ is called \emph{epic} if $g_1\circ f=g_2\circ f$ implies $g_1=g_2$ for any ring homomorphisms $g_1, g_2\colon S\to T$. Any agrarian map $\alpha\colon\ZZ G\to D$ factors in a unique way through an epic agrarian map, namely the map to the skew subfield $D'$ of $D$ generated by $\alpha(\ZZ G)$.
Now any $D'$-module, and in particular the overfield $D$, is flat as a right module over the skew field $D'$. Hence, we get
\begin{align*}
\dim_D H_n(D\otimes C_*)&=\dim_D H_n(D\otimes_{D'} D'\otimes C_*)\\
&=\dim_D D\otimes_{D'} H_n(D'\otimes C_*)\\
&=\dim_{D'} H_n(D'\otimes C_*)
\end{align*}
for any $\ZZ G$-chain complex $C_*$ and $n\in\ZZ$, i.e., $D$ and $D'$ yield the same agrarian Betti numbers. We can thus restrict our attention to epic agrarian maps when computing agrarian Betti numbers.
\end{rema}

\begin{rema}
\label{rema:betti:epic_rational}
Let $G$ be a finitely generated group and $\alpha\colon\ZZ G\to D$ a rational agrarian map for $G$.
If we replace $D$ by the skew subfield generated by $\alpha(\ZZ G)$, then the resulting epic agrarian map will still be rational.
In fact, if we denote the free part of the abelianisation of $G$ by $H$ and the kernel of the projection of $G$ onto $H$ by $K$, then the skew subfield of $D$ generated by $\alpha(\ZZ G)$ is $\Ore(D'H)$, where $D'$ is the skew subfield of $D$ generated by $\alpha(\ZZ K)$.
\end{rema}

The epic agrarian map $\ZZ G\to\ZZ \to \QQ$ given by the augmentation homomorphism always produces positive zeroth agrarian Betti numbers, whereas the zeroth agrarian Betti numbers with respect to any other epic agrarian map is always zero, as we will see in \cref{theo:betti:props:zeroth}.
For agrarian embeddings, the situation is as follows:
\begin{prop}
\label{prop:betti:independence}
Let $G$ be a finitely generated agrarian group. The agrarian Betti numbers for any connected finite free $G$-CW-complex are independent of the choice of agrarian embedding if and only if there exists an epic agrarian embedding for $G$ that is unique up to isomorphism.
\end{prop}
\begin{proof}
If there is a unique isomorphism type of epic agrarian embeddings for $G$, then by the preceding discussion every choice of an agrarian embedding factors through an epic agrarian embedding of this type and hence gives the same agrarian Betti numbers even for all $\ZZ G$-complexes.

Now let $\ZZ G\hookrightarrow D_1$ and $\ZZ G\hookrightarrow D_2$ be non-isomorphic epic agrarian embeddings. By~\cite[Theorem~4.3.5]{Cohn1995} there exists an $m\times n$-matrix $A$ over $\ZZ G$ which becomes invertible when viewed as a matrix over $D_1$, but becomes singular over $D_2$ (without loss of generality). We realise $A$ topologically by constructing the skeleta of a connected finite free $G$-CW-complex $X$ step by step as follows: Choose a generating set $S=\{x_1,\dots,x_k\}$ for $G$ and consider the Cayley graph $C(G, S)$ as a 1-dimensional $G$-CW-complex $X_1$. We now attach $m$ $2$-dimensional spheres to each vertex, and extend the action of $G$ in the obvious way; this way we arrive at the $2$-skeleton $X_2$ of $X$. For every $1\leq j \leq n$, we then attach a $3$-dimensional $G$-cell to $X_2$ in such a way that the resulting boundary map $C_3(X) \to C_2(X)$ in the cellular chain complex of the resulting space coincides with the matrix $A$. This concludes the construction of the connected finite free $G$-CW-complex $X$.

The cellular chain complex $C_*(X)$ is concentrated in degrees 0, 1, 2 and 3 and looks as follows:
\[\ZZ G^n\xrightarrow{A}\ZZ G^m\xrightarrow{0} \ZZ G^k \xrightarrow{(x_1-1, \dots, x_k-1)} \ZZ G\]
As $A$ becomes invertible over $D_1$, but singular over $D_2$, we have
\[H_2(X;D_1)=0 \neq H_2(X;D_2),\]
and hence the $D_1$- and $D_2$-Betti numbers of $X$ differ.
\end{proof}

Translated into our setting, in~\cite[Section~V]{Lewin1974}, Lewin constructs two non-isomorphic epic agrarian embeddings of $F_6$, the free group on six generators. Taking the previous lemma into consideration, we conclude that the notion of \emph{the} agrarian Betti numbers of an $F_6$-CW-complex is not well-defined. Nonetheless, we will later give examples of complexes for which the $D$-Betti numbers can be shown to not depend on $D$.

\subsection{Computational properties} In order to formulate and prove agrarian analogues of the properties of $L^2$-Betti numbers, as collected by Lück in~\cite[Theorem~1.35]{Luck2002}, we have to introduce a few classical constructions on $G$-CW-complexes and chain complexes.

Recall that for a free $G$-CW-complex $X$ and a subgroup $H\leq G$ of finite index, the $H$-space $\res_G^H X$ is obtained from $X$ by restricting the action to $H$. A free (finite, finite type) $H$-CW-structure for this space can be obtained from a free (finite, finite type) $G$-CW-structure of $X$ by replacing a $G$-cell with $\vert G:H\vert $ many $H$-cells.

If $H\leq G$ is any subgroup and $Y$ is a free $H$-CW-complex, then $G\times_H Y$ is the $H$-space $G\times Y/(g, y)\sim (gh^{-1}, hy)$. A free (finite, finite type) $H$-CW-structure of $Y$ determines a free (finite, finite type) $G$-CW-structure of $G\times_H Y$ by replacing an $H$-cell with a $G$-cell.

We now consider a chain complex $C_*$ with differentials $c_*$. Its \emph{suspension} $\Sigma C_*$ is the chain complex with $C_{n-1}$ as the module in degree $n$ and $n$-th differential equal to $-c_{n-1}$. If $f_*\colon C_*\to D_*$ is a chain map between chain complexes with differentials $c_*$ and $d_*$, the \emph{mapping cone} $\cone_*(f_*)$ is the chain complex with $\cone_n(f_*)=C_{n-1}\oplus D_n$ and $n$-th differential given by
\[C_{n-1}\oplus D_n\xrightarrow{\begin{pmatrix}-c_{n-1}&0\\f_{n-1}&d_n\end{pmatrix}} C_{n-2}\oplus D_{n-1}.\]
The mapping cone of $f_*$ fits into the following short exact sequence:
\[0\to D_*\to \cone_*(f_*) \to \Sigma C_*\to 0. \]

The following theorem covers all the properties of agrarian Betti number we will use in computations:
\begin{theo}
\label{theo:betti:props}
The following properties of $D$-Betti numbers hold, where we fix an agrarian map $\alpha\colon \ZZ G\to D$ for a group $G$:

\begin{enumerate}[label={(\arabic*)},ref={\thetheo~(\arabic*)}]
    \item \itemtitle{Homotopy invariance}
    \label[theoenum]{theo:betti:props:invariance}
    Let $f\colon X\to Y$ be a $G$-map of free $G$-CW-complexes of finite type. If the map $H_p(f;\ZZ)\colon H_p(X;\ZZ)\to H_p(Y;\ZZ)$ induced on cellular homology with integral coefficients is bijective for $p\leq d-1$ and surjective for $p=d$, then
    \begin{align*}
        \bag_p(X) &= \bag_p(Y)\quad \text{for }p\leq d-1;\\
        \bag_d(X) &\geq \bag_d(Y).
    \end{align*}
    In particular, if $f$ is a weak homotopy equivalence, we get for all $p\geq 0$:
    \[\bag_p(X)=\bag_p(Y).\]

    \item \itemtitle{Euler-Poincaré formula}
    \label[theoenum]{theo:betti:props:euler}
    Let $X$ be a finite free $G$-CW-complex. Let $\chi(X/G)$ be the Euler characteristic of the finite CW-complex $X/G$, i.e.,
    \[\chi(X/G)\coloneqq \sum_{p\geq 0} (-1)^p\cdot \beta_p(X/G), \]
    where $\beta_p(X/G)$ denotes the number of $p$-cells of $X/G$. Then
    \[\chiag(X)\coloneqq\sum_{p\geq 0}(-1)^p\cdot \bag_p(X) = \chi(X/G).\]

    \item \itemtitle{Upper bound}
    \label[theoenum]{theo:betti:props:bound}
    Let $X$ be a free $G$-CW-complex. With $\beta_p(X/G)$ as above, for all $p\geq 0$ we have
    \[\bag_p(X)\leq \beta_p(X/G).\]

    \item \itemtitle{Zeroth agrarian Betti number}
    \label[theoenum]{theo:betti:props:zeroth}
    Let $X$ be a connected free $G$-CW-complex of finite type and assume that the agrarian map $\alpha\colon\ZZ G\to D$ does not factor through the augmentation homomorphism $\ZZ G\to\ZZ\hookrightarrow\QQ$.
    Then
    \[\bag_0(X)=0.\]

    \item \itemtitle{Induction}
    \label[theoenum]{theo:betti:props:induction}
    Let $H\leq G$ be a subgroup of $G$.
    If $X$ is a free $H$-CW-complex, then for $p\geq 0$
    \[\bag_p(G\times_H X)=\bag_p (X),\]
    where the agrarian embedding for $H$ is chosen as the restriction of $\alpha$ to $\ZZ H$.

    \item \itemtitle{Amenable groups}
    \label[theoenum]{theo:betti:props:amenable}
    Let $X$ be a free $G$-CW-complex of finite type and assume that the agrarian map $\alpha\colon \ZZ G\to D$ is actually an agrarian embedding.
    Further assume that $G$ is amenable.
    Then
    \[\bag_p(X) = \dim_D(D\otimes H_p(X;\ZZ G)).\]
\end{enumerate}
\end{theo}
\begin{proof}
\begin{enumerate}[label={(\arabic*)}]
    \item We replace $f$ by a homotopic cellular map. Consider the $\ZZ G$-chain map \[f_*\colon C_*(X)\to C_*(Y)\]
    induced by $f$ on the cellular chain complexes and its mapping cone $\cone_*(f_*)$, which fits into a short exact sequence
    \[0\to C_*(Y)\to \cone_*(f_*)\to \Sigma C_*(X)\to 0\]
     of $\ZZ G$-chain complexes. Applying the assumptions on the map $H_p(f;\ZZ)$ to the long exact sequence in homology associated to this short exact sequence, we obtain that $H_p(\cone_*(f_*))=0$ for $p\leq d$.
    \begin{clai}
        The homology of $D\otimes\cone_*(f_*)$ vanishes in degrees $p\leq d$.
    \end{clai}

    Assume for the moment that this indeed holds. Since $D\otimes\cone_*(f_*)= \cone_*(\id_D\otimes f_*)$ and $D\otimes\Sigma C_*(X)=\Sigma(D\otimes C_*(X))$, the short sequence
    \[0\to D\otimes C_*(Y)\to D\otimes\cone_*(f_*)\to D\otimes\Sigma C_*(X)\to 0\]
    is also exact. We now consider the associated long exact sequence in homology, in which the terms $H_p(D\otimes\cone_*(f_*))$ for $p\leq d$ vanish by the claim. The exactness of the sequence then implies that the differentials $H_p(D\otimes\Sigma C_*(Y))\xrightarrow{\cong} H_{p-1}(D\otimes C_*(X))$ are isomorphisms for $p\leq d$ and the differential $H_{d+1}(D\otimes\Sigma C_*(Y))\twoheadrightarrow H_d(D\otimes C_*(X))$ is an epimorphism. Applying $\dim_D$ and using the definition of the suspension then yields the desired statement.

    We are left with proving the claim. Since $\cone_*(f_*)$ is bounded below and consists of free modules, we can inductively construct a $\ZZ G$-chain homotopy equivalent $\ZZ G$-chain complex $Z_*$ which vanishes in degrees $p\leq d$. Tensoring with $D$ then yields a $D$-chain homotopy equivalence between $D\otimes\cone_*(f_*)$ and $D\otimes Z_*$. As $Z_p=0$ for $p\leq d$, the same holds true for $D\otimes Z_*$ and hence $H_p(D\otimes\cone_*(f_*))=H_p(D\otimes Z_*)=0$ for $p\leq d$.

    \item This is a consequence of two immediate facts: first, the Euler characteristic of a chain complex over a skew field does not change when passing to homology; second, we have the identity $\beta_p(X/G)=\dim_D D\otimes C_p(X)$.

    \item This holds since $H_p(X;D)$ is a subquotient of $D\otimes C_p(X)$ and the latter has dimension $\beta_p(X/G)$ over $D$ (as remarked above).

    \item If $X$ is empty, then the claim is trivially true. Otherwise, we will first argue that, without loss of generality, we may assume $X/G$ to have exactly one $0$-cell.  let $T$ be a maximal tree in the 1-skeleton of the CW-complex $X/G$ and denote by $q\colon X/G\to (X/G)/T$ the associated cellular quotient map, which is a homotopy equivalence. Note that $(X/G)/T$ has a single $0$-cell. Let $p\colon (X/G)/T \to X/G$ be a cellular homotopy inverse of $q$. We denote by $X'$ the total space in the following pullback of the $G$-covering $X\to X/G$ along $p$:
    \[
        \begin{tikzcd}
            X \arrow[d] & X' \arrow[d, dashed]\arrow[l, dashed]\\
            X/G \arrow[r, "q"] & (X/G)/T \arrow[l, "p", bend left]
        \end{tikzcd}
    \]
    Alternatively, we can view $X'$ as being obtained from $X$ by collapsing each lift of $T$ individually to a point.
    Since $(X/G)/T$ is a connected free $G$-CW-complex of finite type, $X\to X/G$ is a $G$-covering and $X$ is connected, the $G$-CW-complex $X'$ is also connected, free and of finite type. Furthermore, $X'$ is $G$-homotopy equivalent to $X$ via any $G$-equivariant lift of the homotopy equivalence $p$. By \cref{theo:betti:props:invariance}, the $D$-Betti numbers of $X$ and $X'$ agree, so we may assume without loss of generality that $X$ has a single equivariant 0-cell.

    Since $X$ is a free $G$-CW-complex of finite type which has a single $0$-cell, the differential $c_1\colon C_1(X)\to C_0(X)$ in its cellular chain complex is of the form
    \[\ZZ G^n \xrightarrow{\oplus_{i=1}^n (1-g_i)} \ZZ G\]
    for $g_i\in G, i=1,\dots,n, n\in \NN$ for any choice of a $\ZZ G$-basis of $C_*(X)$ consisting of cells. The image of the differential is contained in the augmentation ideal $I=\langle g-1\mid g\in G\rangle$ of $\ZZ G$, and, as $X$ is assumed to be connected, has to coincide with it for $H_0(X;\ZZ)$ to be isomorphic to $\ZZ$. By our additional assumption on the agrarian map, there is thus an element in the image of the differential that does not lie in the kernel of the agrarian map. But the image of this element is invertible in $D$, and hence $H_0(X;D)=0$ as claimed.

    \item On cellular chain complexes, $G\times_H ?$ translates into applying the functor $\ZZ G\otimes_{\ZZ H}?$.
    The claim thus follows from the canonical identification
    \[D\otimes_{\ZZ G}\ZZ G\otimes_{\ZZ H} C_*(X)\cong D\otimes_{\ZZ H} C_*(X).\]

    \item As $G$ is agrarian, its group ring $\QQ G$ does not admit zero divisors (this is immediate, since $\QQ G$ embeds into the same skew field $D$ as $\ZZ G$ does). A result of Tamari~\cite{Tamari1954} we used already for free abelian groups now implies that since $G$ is amenable, it admits an Ore field of fractions $F$ of $\QQ G$, and hence of $\ZZ G$. In particular, $F$ is flat over $\ZZ G$ and every embedding of $\ZZ G$ into a skew field, such as the agrarian embedding $\alpha\colon \ZZ G\hookrightarrow D$, factors through the natural inclusion $\ZZ G\hookrightarrow F$. We thus obtain the following for any $p\geq 0$, using first that $F\hookrightarrow D$ is flat and then that $\ZZ G\hookrightarrow F$ is flat:
    \begin{align*}
      \bag_p(X) &= \dim_D H_p(X;D) = \dim_D D\otimes_F H_p(X;F)\\
      &= \dim_D D\otimes_F F\otimes H_p(X;\ZZ G) = \dim_D D\otimes H_p(X;\ZZ G). \qedhere
    \end{align*}
\end{enumerate}
\end{proof}

The behavior of $L^2$-Betti numbers under restriction to finite-index subgroups carries over to agrarian invariants under an additional assumption on the agrarian map:
\begin{prop}
    \label{prop:betti:restriction}
    Let $H\leq G$ be a subgroup of $G$ of finite index $\vert G:H\vert<\infty$.
    Let $\alpha\colon \ZZ G\to D$ be an epic agrarian map for $G$ and denote the skew subfield of $D$ generated by $\alpha(\ZZ H)$ by $D'$.
    Assume that the map
    \[\Psi\colon D'\otimes_{\ZZ H} \ZZ G\to D, x \otimes g\mapsto x\cdot \alpha(g^{-1})\]
    of $D'$-$\ZZ G$-bimodules is an isomorphism.
    If $X$ is a free $G$-CW-complex of finite type, then for $p\geq 0$
    \[\bag_p(\res^H_G X)=|G:H|\cdot \bag_p(X).\]
\end{prop}
\begin{proof}
    Since $D'\otimes_{\ZZ H} \ZZ G \otimes_{\ZZ G} C_*(X)\cong D'\otimes_{\ZZ H} C_*(\res_G^H X)$, the map $\Psi$ induces an isomorphism
    \[D'\otimes_{\ZZ H} C_*(\res_G^H X) \xrightarrow{\cong} \res_{D}^{D'} D \otimes_{\ZZ G} C_*(X)\]
    of $D'$-chain complexes.
    Passing to agrarian Betti numbers on both sides, we obtain
    \begin{equation}
        \label{eq:betti:restriction}
        \bag_p(\res_G^H X)=\dim_{D'} \res_{D}^{D'} H_p(D\otimes_{\ZZ G} C_*(X)).
    \end{equation}
    Since $\ZZ G$ is a free left $\ZZ H$-module of rank $\vert G:H\vert$, the isomorphism $\Psi$ exhibits $D$ as a left $D'$-vector space of dimension $\vert G:H\vert$.
    As a consequence,
    \[\dim_{D'} \res_{D}^{D'} V=\vert G:H\vert \cdot \dim_{D} V\]
    holds for any left $D$-vector space $V$.
    We arrive at the claimed formula by applying this identity to the right-hand side of \eqref{eq:betti:restriction}.
\end{proof}

\subsection{Mapping tori} In subsequent sections, we will study invariants of CW complexes with vanishing agrarian Betti numbers. In the context of $L^2$-invariants, an extremely useful way of showing the vanishing of $L^2$-Betti numbers comes from a celebrated theorem of Lück~\cite[Theorem~2.1]{Luck1994}. Below, we offer a straightforward adaption of Lück's result to the setting of agrarian Betti numbers. If $G$ satisfies the Atiyah conjecture, then our version reduces to the classical $L^2$-formulation if one considers the agrarian embedding $\ZZ G\hookrightarrow \linnelld(G)$ into the Linnell skew field.

\begin{defi}
    Let $f\colon X\to X$ be a selfmap of a path-connected space. The \emph{mapping torus} $T_f$ of $f$ is obtained from the cylinder $X\times[0, 1]$ by identifying $(x, 1)$ with $(f(x), 0)$ for every $x\in X$. The \emph{canonical projection} is the map $T_f\to S^1$ sending $(x, t)$ to $\exp(2\pi it)$. It induces an epimorphism $\pi_1(T_f)\to\pi_1(S^1)=\ZZ$.
\end{defi}

If $X$ has the structure of a CW-complex with $\beta_p(X)$ cells of dimension $p$ and $f$ is cellular, then $T_f$ can be endowed with a CW-structure with $\beta_p(T_f)=\beta_p(X)+\beta_{p-1}(X)$ cells of dimension $p$ for each $p\geq 0$.

\begin{theo}
\label{theo:betti:torus}
Let $f\colon X\to X$ be a cellular selfmap of a connected CW-complex $X$ and $\pi_1(T_f)\xrightarrow{\phi}G\xrightarrow{\psi}\ZZ$ any factorisation into epimorphisms of the epimorphism induced by the canonical projection.
Let $\covering{T_f}$ be the covering of the mapping torus $T_f$ associated to $\phi$, endowed with the structure of a connected free $G$-CW-complex.
Let $\alpha\colon \ZZ G\to D$ be a rational agrarian map for $G$.
If the $d$-skeleton of $X$ (and thus of $\covering{T_f}$) is finite for some $d\geq 0$, then for all $p\leq d$
\[\bag_p\left(\covering{T_f}\right)=0.\]
\end{theo}
\begin{proof}
The topological part of the proof is the analogue of the proof for $L^2$-Betti numbers, see~\cite[Theorem~1.39]{Luck2002}.

By \cref{rema:betti:epic,rema:betti:epic_rational}, we may assume that $\alpha$ is epic.
Fix $p\geq 0$. For any $n\geq 1$, define $G_n\leq G$ to be the preimage of the subgroup $n\cdot \ZZ\leq \ZZ$ under $\psi\colon G\to \ZZ$, for which we consider the induced agrarian map $\ZZ G_n\hookrightarrow \ZZ G\xrightarrow{\alpha} D$.
Further denote the kernel of $\psi$ by $K$ and the skew subfield of $D$ generated by $\ZZ K$ by $D'$.
Since the agrarian map $\alpha$ is epic and rational, the skew subfield $D_n$ of $D$ generated by $\alpha(\ZZ G_n)$ is given by $\Ore(D'(G_n/K))$.
\begin{clai}
    For our choice of $\alpha\colon \ZZ G\to D$ and $H\coloneqq G_n$, the map $\Psi$ of \cref{prop:betti:restriction} is an isomorphism.
\end{clai}
We first conclude the proof assuming the claim. Since $G_n$ has index $n$ in $G$, we deduce from the claim and \cref{prop:betti:restriction} that
\begin{equation}
    \label{eq:betti:torus:transfer}
    \bag_p\left(\covering{T_f}\right)=\frac{1}{n}\cdot\bag_p\left(\res_G^{G_n} \covering{T_f}\right).
\end{equation}
Reparametrising yields a homotopy equivalence $h\colon T_{f^n}\xrightarrow{\simeq}\covering{T_f}/G_n$ of CW-complexes, where $f^n$ denotes the $n$-fold composition of $f$. Let $\covering{T_{f^n}}$ be the $G_n$-space obtained as the following pullback, or equivalently, as the covering of $T_{f^n}$ corresponding to the kernel of $\pi_1(T_{f^n})\cong\pi_1(\covering{T_f}/G_n)\to G_n$:
\[
\begin{tikzcd}
    \covering{T_{f^n}}\arrow[r, "\covering{h}"]\arrow[d]&\res_G^{G_n}\covering{T_f}\arrow[d]\\
    T_{f^n}\arrow[r, "h"] & \covering{T_f}/G_n
\end{tikzcd}
\]
Since $h$ is a homotopy equivalence between base spaces of $G_n$-coverings, $\covering{h}$ is a $G_n$-homotopy equivalence.
By \cref{theo:betti:props:invariance}, we obtain
\begin{equation}
    \label{eq:betti:torus:homotopy}
    \bag_p\left(\covering{T_{f^n}}\right) = \bag_p\left(\res_G^{G_n}\covering{T_f}\right)
\end{equation}
for $p\geq 0$. Since $T_{f^n}$ has a CW-structure with $\beta_p(X)+\beta_{p-1}(X)$ cells of dimension $p$ and this number is finite by assumption, using \cref{theo:betti:props:bound} we conclude:
\[\bag_p\left(\covering{T_f}\right)\stackrel{\eqref{eq:betti:torus:transfer}}{=}\frac{1}{n}\cdot\bag_p\left(\res_G^{G_n}\covering{T_f}\right)\stackrel{\eqref{eq:betti:torus:homotopy}}{=}\frac{1}{n}\cdot\bag_p\left(\covering{T_{f^n}}\right)\leq \frac{\beta_p(X)+\beta_{p-1}(X)}{n}.\]
Letting $n\to\infty$ finishes the proof of the theorem assuming the claim.

\smallskip
\noindent \textit{Proof of the claim.}
For this proof, it is instructive to reinterpret the objects we are dealing with. Recall that $D = \Ore(D'(G/K))$, and hence its elements are twisted rational functions in one variable, say $t$, with coefficients in $D'$. Similarly, $D_n$ consists of such rational functions in a single variable $t^n$, and the embedding $D_n\to D$ is obtained by identifying the variable $t^n$ in the former ring of rational functions with the $n^{th}$ power of $t$ in the latter (as the notation suggests).

Now it becomes clear that $D_n \otimes_{\ZZ {G_n}} \ZZ G$ is generated by elements of the form $pq^{-1}\otimes t^m$ where $m \in \{0, \dots, n-1\}$ and where $p,q$ are twisted polynomials in $t^n$ with $q \neq 0$. Therefore we may view $D_n\otimes_{\ZZ {G_n}} \ZZ G$ as consisting of elements of the form $pq^{-1}$ where $q$ is a non-zero polynomial in $t^n$, and $p$ is a polynomial in $t$. Viewed in this way, the map $\Psi\colon D_n \otimes_{\ZZ {G_n}} \ZZ G \to D$ maps identically into $D$.

We are left to see that $\Psi$ is surjective, which we will achieve by equipping its domain with a ring structure.
If we denote the cyclic group $G/G_n$ of order $n$ by $\ZZ_n$, then $D_n\otimes_{\ZZ {G_n}} \ZZ G$ is identified with the twisted group ring $D_n \ZZ_n$ via the map $pq^{-1}\otimes t^m\mapsto pq^{-1} \ast m$, where $m \in \{0, \dots, n-1\}$ and $p,q$ are twisted polynomials in $t^n$ with $q \neq 0$.
We can thus replace the domain of $\Psi$ with $D_n \ZZ_n$ and note that the resulting map, which we again denote by $\Psi$, is in fact an injective ring homomorphism.
Since $\ZZ_n$ is a finite group and $D_n$ is a skew field of characteristic 0, the twisted group ring $D_n \ZZ_n$ is semisimple by~\cite[Lemma~10.55]{Luck2002} -- note that this is a version of Maschke's theorem for twisted group rings.
Since a semisimple subring of a skew field is a skew field and $D$ is assumed to be generated by $\ZZ G\subset D_n \ZZ_n$, we conclude that $\Psi$ is also surjective and hence an isomorphism.
\end{proof}

\section{Agrarian Torsion}
Having introduced agrarian Betti numbers together with computational tools allowing us to prove their vanishing for certain spaces, we will now present a secondary invariant for such spaces. This invariant will be called \emph{agrarian torsion} and arises as Reidemeister torsion with values in the abelianised units of the skew field $D$. It is motivated by the construction of \emph{universal $L^2$-torsion} by Friedl and Lück~\cite{FL2017}. We will reference the rather general treatment of torsion by Cohen~\cite{Cohen1973} throughout this section.

As usual, in this section $G$ will always be a group with a fixed agrarian map $\alpha\colon \ZZ G\to D$.

\subsection{\texorpdfstring{Contractible $D$-chain complexes}{Contractible D-chain complex}}
In order to define agrarian torsion, we require a contractible $D$-chain complex. In our case, contractibility is governed by the agrarian Betti numbers because of
\begin{prop}[{\cite[Proposition~1.7.4]{Rosenberg1994}}]
\label{prop:torsion:contractible}
Let $R$ be a ring and $C_*$ an $R$-chain complex. If $C_*$ is acyclic, vanishes in sufficiently small degree and consists of projective $R$-modules, then $C_*$ is contractible.
\end{prop}

\begin{lemm}
\label{lemm:torsion:acyclicity}
A finite $\ZZ G$-chain complex $C_*$ is $D$-acyclic if and only if $D\otimes C_*$ is contractible.
\end{lemm}
\begin{proof}
Since $C_*$ is finite, the $D$-chain complex $D\otimes C_*$ is in particular bounded below. All its modules are free because $D$ is a skew field, and hence the statement follows from \cref{prop:torsion:contractible}.
\end{proof}

Agrarian torsion, being constructed as non-commutative Reidemeister torsion, naturally takes values in the first $K$-group of $D$:
\begin{defi}
Let $R$ be a ring. Denote by $\GL(R)$ the direct limit of the groups $\GL_n(R)$ of invertible $n\times n$ matrices over $R$ with the embeddings given by adding an identity block in the bottom-right corner. The \emph{$K_1$-group} $K_1(R)$ is defined as the abelianisation of $\GL(R)$. The \emph{reduced $K_1$-group} $\Kr(R)$ is defined as the quotient of $K_1(R)$ by the subgroup $\{(\pm 1)\}$.
\end{defi}

We now consider a $D$-acyclic finite free $\ZZ G$-chain complex $(C_*, c_*)$. Such a complex will be called \emph{based} if it comes with a choice of preferred bases for all chain modules. By the previous lemma, we can find a chain contraction $\gamma_*$ of $D\otimes C_*$. Set $C_{\mathrm{odd}}\coloneqq \bigoplus\limits_{i \, \mathrm{ odd}} C_i$ and $C_{\mathrm{even}}\coloneqq \bigoplus\limits_{i \, \mathrm{ even}} C_i$. Note that $D$-acyclicity guarantees that $\dim_D D\otimes C_\mathrm{odd} = \dim_D D\otimes C_\mathrm{even}$.

\begin{lemm}
\label{lemm:torsion:construction}
In the situation above, the map $c_*+\gamma_*\colon D\otimes C_{\mathrm{odd}}\to D\otimes C_{\mathrm{even}}$ is an isomorphism of finitely generated based free $D$-modules and the class in $\Kr(D)$ defined by the matrix representing it in the preferred basis does not depend on the choice of $\gamma$.
\end{lemm}
\begin{proof}
That the map is an isomorphism is the content of~\cite[(15.1)]{Cohen1973}, the independence is covered by~\cite[(15.3)]{Cohen1973}.
\end{proof}

\subsection{The Dieudonné determinant} The $K_1$ groups of skew fields can be determined using a generalisation of the classical determinant of a matrix over a field to matrices over skew fields, which is known as the \emph{Dieudonné determinant}. As opposed to the situation for fields, there is no longer a polynomial expression in terms of the entries of the matrix; rather, the Dieudonné determinant is defined by an inductive procedure:

\begin{defi}
Let $A=(a_{ij})$ be an $n\times n$ matrix over a skew field $D$. The \emph{canonical representative of the Dieudonné determinant} $\det^c A\in D$ is defined inductively as follows:
\begin{enumerate}[label=(\arabic*)]
    \item If $n=1$, then $\det^c A\coloneqq a_{11}$.
    \item If the last row of $A$ consists of zeros only, then $\det^c A \coloneqq 0$.
    \item If $a_{nn}\neq 0$, then we form the $(n-1)\times(n-1)$ matrix $A'=(a_{ij}')$ by setting $a_{ij}'\coloneqq a_{ij} - a_{in}a_{nn}^{-1}a_{nj}$, and declare $\det^c A\coloneqq\det^c A'\cdot a_{nn}$.
    \item Otherwise, let $j<n$ be maximal such that $a_{nj}\neq 0$. Let $A'$ be obtained from $A$ by interchanging rows $j$ and $n$. Then set $\det^c A\coloneqq-\det^c A'$.
\end{enumerate}
The \emph{Dieudonné determinant} $\det A$ of $A$ is defined to be the image of $\det^c A$ in $D^\times/[D^\times, D^\times] \sqcup \{0\}$, i.e., in the abelianised unit group of $D$ if $\det^c A\neq 0$. We also write $\Dab$ for this group.
\end{defi}

As a convention, we will write the group operation of the abelian group $\Dab$ (and its quotients) additively.

If $D$ is a commutative field, then the Dieudonné determinant coincides with the usual determinant as the matrix $A$ is brought into upper-diagonal form during the inductive procedure defining $\det^c A$.

The Dieudonné determinant is multiplicative on all matrices and takes non-zero values on invertible matrices~\cite{Dieudonne1943}.
\begin{prop}[{\cite[Corollary~2.2.6]{Rosenberg1994}}]
Let $D$ be a skew field. Then the Dieudonné determinant $\det\colon\GL(D)\to\Dab$ induces group isomorphisms
\[\det\colon K_1(D)\xrightarrow{\cong} \Dab \textrm{ and}\]
\[\det\colon \Kr(D)\xrightarrow{\cong} \Dab/\{\pm 1\}.\]
\end{prop}

\subsection{Definition and properties of agrarian torsion} Relying on the explicit description of $\Kr(D)$ obtained above, we can motivate
\begin{defi}
\label{defi:torsion:torsion}
The \emph{$D$-agrarian torsion} of a $D$-acyclic finite based free $\ZZ G$-chain complex $(C_*, c_*)$ is defined as
\[\torsionag(C_*)\coloneqq\det([c_*+\gamma_*])\in \Dab/\{\pm 1\},\]
where $[c_*+\gamma_*]\in \Kr(D)$ is the class determined by the (representing matrix of the) isomorphism constructed in \cref{lemm:torsion:construction}.
\end{defi}

The usual additivity property for torsion invariants directly carries over to the agrarian setting in the following form:
\begin{lemm}[{\cite[(17.2)]{Cohen1973}}]
\label{lemm:torsion:additivity}
Let $0\to C'_*\to C_*\to C''_*\to 0$ be a short exact sequence of finite based free $\ZZ G$-chain complexes such that the preferred basis of $C_*$ is composed of the preferred basis of $C'_*$ and preimages of the preferred basis elements of $C''_*$. Assume that any two of the complexes are $D$-acyclic. Then so is the third and
\[\torsionag(C_*)=\torsionag(C'_*)+\torsionag(C''_*).\]
\end{lemm}

The difference in agrarian torsion between $\ZZ G$-chain homotopy equivalent chain complexes is measured by the Whitehead torsion of the chain homotopy equivalence, analogously to the statement of~\cite[Lemma~2.10]{FL2017} for universal $L^2$-torsion:
\begin{lemm}
\label{lemm:torsion:invariance}
Let $f\colon C_*\to E_*$ be a $\ZZ G$-chain homotopy equivalence of finite based free $\ZZ G$-chain complexes. Denote by $\rho(\cone(f_*))\in \Kr(\ZZ G)$ the Whitehead torsion of the contractible finite based free $\ZZ G$-chain complex $\cone(f_*)$. If one of $C_*$ and $E_*$ is $D$-acyclic, then so is the other and we get
\[\torsionag(E_*)-\torsionag(C_*)=\det\nolimits_D\Big(\alpha_*\big(\rho(\cone(f_*))\big)\Big),\]
where $\alpha_*\colon \Kr(\ZZ G)\to \Kr(D)$ is induced by $\alpha\colon \ZZ G\hookrightarrow D$.
\end{lemm}
\begin{proof}
Since $f_*$ is a $\ZZ G$-chain homotopy equivalence, the finite free $\ZZ G$-chain complex $\cone(f_*)$ is contractible and hence its Whitehead torsion $\rho(\cone(f_*))$ is defined.
The finite free $D$-chain complex $D\otimes\cone(f_*)$ is again contractible and since the matrix defining its agrarian torsion are already invertible over $\ZZ G$, we get that $\torsionag(\cone(f_*))=\det\nolimits_D(\alpha_*(\rho(\cone(f_*))))$.

We now apply \cref{lemm:torsion:additivity} to the short exact sequence
\[0\to E_*\to\cone_*(f_*)\to\Sigma C_*\to 0\]
 with $\cone_*(f_*)$ and one of $\Sigma C_*$ and $E_*$ being $D$-acyclic.
Since $\torsionag(\Sigma C_*)=-\torsionag(C_*)$, as is readily observed from the definition of $\torsionag$, we obtain that $\torsionag(E_*)-\torsionag(C_*)=\torsionag(\cone(f_*))=\det\nolimits_D(\alpha_*(\rho(\cone(f_*))))$.
\end{proof}

Our goal is to apply the concept of $D$-agrarian torsion to $G$-CW-complexes. Since the free cellular chain complexes associated to such complexes do not admit a canonical $\ZZ G$-basis, but only a canonical $\ZZ$-basis (up to orientation), we have to account for this additional indeterminacy by passing to a further quotient of $\Dab$:
\begin{defi}
Let $X$ be a $D$-acyclic finite free $G$-CW-complex. The \emph{$D$-agrarian torsion} of $X$ is defined as
\[\torsionag(X)\coloneqq \torsionag(C_*(X)) \in \Dab/\{\pm g\mid g\in G\},\]
where $C_*(X)$ is endowed with any $\ZZ G$-basis that projects to a $\ZZ$-basis of $C_*(X/G)$ consisting of unequivariant cells.
\end{defi}

That $\torsionag(X)$ is indeed well-defined can be seen from~\cite[(15.2)]{Cohen1973}.

\subsection{\texorpdfstring{Comparison with universal $L^2$-torsion}{Comparison with universal L\^{}2-torsion}}
A rich source of agrarian groups is the class of torsion-free groups that satisfy the Atiyah conjecture. For these groups, there is a canonical skew field $\linnelld(G)$ in which the group ring $\ZZ G$ embeds. In the case of $D=\linnelld(G)$, agrarian torsion coincides with the determinant of the universal $L^2$-torsion introduced by Friedl and Lück in~\cite{FL2017}, as we will see now.

Universal $L^2$-torsion naturally lives in a weak version of the $K_1$-group of the group ring, which is defined as follows:
\begin{defi}[{\cite[Definition~2.1]{FL2017}}]
Let $G$ be a group. Denote by $\Kw(\ZZ G)$ the \emph{weak $K_1$-group}, which is defined to be an abelian groups with the following generators and relations:
\begin{description}
    \item[Generators] $[A]$ for square matrices $A$ over $\ZZ G$ that become invertible after the change of rings $\ZZ G\hookrightarrow\linnelld(G)$
    \item[Relations]
    \begin{itemize}
        \item $[AB] = [A] + [B]$ for matrices $A$ and $B$ of compatible sizes and such that $A$ and $B$ become invertible over $\linnelld(G)$.
        \item $[D] = [A] + [C]$ for a block matrix
        \[D=\begin{pmatrix}A & B \\ 0 & C\end{pmatrix}\]
        with $A$ and $C$ square and invertible over $\linnelld(G)$.
    \end{itemize}
\end{description}
Define the \emph{weak Whitehead group} $\Whw(G)$ as the quotient of $\Kw(\ZZ G)$ by the subgroup generated by the $1\times 1$-matrices $(\pm g)$ for all $g\in G$.
\end{defi}

Note that there are canonical maps $K_1(\ZZ G)\to \Kw(\ZZ G)$ and $\Kw(\ZZ G)\to K_1(\linnelld(G))$ given by $[A]\mapsto [A]$ and $[A]\mapsto [1\otimes A]$ on generators, respectively.

The following result by Linnell and Lück indicates that the abelian groups in which agrarian torsion and universal $L^2$-torsion take values coincide up to isomorphism for a large class of groups:
\begin{theo}[\cite{LL2018}]
\label{theo:torsion:linnell_luck}
Let $\mathcal{C}$ be the smallest class of groups which contains all free groups and is closed under direct union and extensions by elementary amenable groups. Let $G$ be a torsion-free group which belongs to $\mathcal{C}$. Then $
\linnelld(G)$ is a skew field and the composite map
\[\Kw(\ZZ G)\to K_1(\linnelld(G)) \xrightarrow{\det} \Dab[\linnelld(G)]\]
is an isomorphism.
\end{theo}

Let $X$ be a finite free $G$-CW-complex that is \emph{$L^2$-acyclic}, i.e., whose $L^2$-Betti numbers vanish. Friedl and Lück~\cite[Definition~3.1]{FL2017} associate to such a $G$-CW-complex an element $\torsionu(X)\in \Whw(G)$ called the \emph{universal $L^2$-torsion} of $X$. We can obtain from this an element
\[\det(\torsionu(X))\in\Dab[\linnelld(G)]/\{\pm g\mid g\in G\},\]
which by \cref{theo:torsion:linnell_luck} carries an equivalent amount of information as $\torsionu$ for many groups $G$.

The statement of the following theorem is implicit in~\cite[Section~2.3]{FLT2016} by Friedl, Lück and Tillmann.

\begin{theo}
\label{theo:torsion:linnelld}
Let $G$ be a torsion-free group that satisfies the Atiyah conjecture. Then $G$ is $\linnelld(G)$-agrarian. Furthermore, if $X$ is any finite free $G$-CW-complex, then $X$ is $\linnelld(G)$-acyclic if and only if it is $L^2$-acyclic. If this is the case, we have
\[\torsionag[\linnelld(G)](X)=\det(\torsionu(X))\in\Dab[\linnelld(G)]/\{\pm g\mid g\in G\}.\]
\end{theo}
\begin{proof}
During the proof, we will use the notion of universal $L^2$-torsion for $L^2$-acyclic finite based free $\ZZ G$-chain complexes as defined in~\cite[Definition~2.7]{FL2017}. The universal $L^2$-torsion of a finite free $G$-CW-complex is then obtained as the universal $L^2$-torsion of the associated cellular chain complex together with any basis consisting of $G$-cells. We will also abuse notation in that we consider classes in $\Kwr(\ZZ G)$ to be represented by both square matrices over $\ZZ G$ (our convention) and $\ZZ G$-endomorphisms of some $\ZZ G^n$ (the convention in~\cite{FL2017}).

The first statement is proved analogously to one direction of~\cite[Lemma~10.39]{Luck2002}, the second statement then follows from \cref{lemm:torsion:acyclicity} and~\cite[Lemma~2.21]{FL2017}.

In order to prove the last statement, we want to make use of the universal property of universal $L^2$-torsion (see~\cite[Remark~2.16]{FL2017}). To this end, we first consider $\ZZ G$-chain complexes of the following simple form: Let $[A]\in\Kwr(\ZZ G)$ be represented by an $n\times n$ matrix $A$ over $\ZZ G$, and let $C_*^A$ be the $\ZZ G$-chain complex concentrated in degrees 0 and 1 with the only non-trivial differential given by $r_A\colon\ZZ G^n\to \ZZ G^n, x \mapsto x\cdot A$. Since $A$ becomes an isomorphism over $\linnelld(G)$, such a complex is always $\linnelld(G)$- and thus $L^2$-acyclic.

The universal $L^2$-torsion of $C_*^A$ is computed from a weak chain contraction $(\delta_*, v_*)$ of $C_*^A$ as defined in~\cite[Definition~2.4]{FL2017}. In this particular case, we can take $\delta_0=\id_{\ZZ G^n}, \delta_p=0$ for $p\neq 0$ and $v_0=v_1=r_A, v_p=0$ for $p\not\in\{0, 1\}$. According to~\cite[Definition~2.7]{FL2017}, the universal $L^2$-torsion of $C_*^A$ is thus given by
\[\torsionu(C_*^A)=[v_1\circ r_A + 0] - [v_1]=[r_A^2]-[r_A]=[A]\in \Kwr(\ZZ G)\]
and hence $\det(\torsionu(C_*^A))=\det A \in \Dab[\linnelld(G)]/\{\pm 1\}$.

The $\linnelld(G)$-agrarian torsion of $C_*^A$ is computed from a (classical) chain contraction of $\linnelld(G)\otimes C_*^A$; let $\gamma_*$ be such a contraction with $\gamma_0=(\id_{\linnelld(G)}\otimes r_A)^{-1}$ and $\gamma_p=0$ for $p\neq 0$. Since $\gamma$ vanishes in odd degrees, the construction of $\linnelld(G)$-agrarian torsion yields
\[\torsionag[\linnelld(G)](C_*^A)=\det([ \id_{\linnelld(G)}\otimes r_A + 0])=\det A\in \Dab[\linnelld(G)]/\{\pm 1\},\]
and hence $\det(\torsionu(C_*^A))=\torsionag[\linnelld(G)](C_*^A)$.

The pair $(\Dab[\linnelld(G)]/\{\pm 1\}, \torsionag[\linnelld(G)])$ consists of an abelian group and an assignment that associates to a $\linnelld(G)$-acyclic (i.e., $L^2$-acyclic) finite based free $\ZZ G$-chain complex an element $\torsionag[\linnelld(G)]\in \Dab[\linnelld(G)]/\{\pm 1\}$. The assignment is additive by \cref{lemm:torsion:additivity} and maps complexes of the shape $\ZZ G\xrightarrow{\pm \id_{\ZZ G}}\ZZ G$ to $1\in \Dab[\linnelld(G)]/\{\pm 1\}$ by construction. It hence constitutes an example of an \emph{additive $L^2$-torsion invariant} in the sense of~\cite[Remark~2.16]{FL2017}. Since by~\cite[Theorem~2.12]{FL2017} the pair $(\Kwr(\ZZ G), \torsionu)$ is the universal such invariant, there is a unique group homomorphism $f\colon \Kwr(\ZZ G)\to \Dab[\linnelld(G)]/\{\pm 1\}$ satisfying $f\circ \torsionu = \torsionag[\linnelld(G)]$.

It is left to check that $f$ and $\det$ agree as maps $\Kwr(\ZZ G)\to \Dab[\linnelld(G)]/\{\pm 1\}$. We have seen already that $\det(\torsionu(C_*^A))=\torsionag[\linnelld(G)](C_*^A)$. But $\torsionu(C_*^A)=[A]$, and hence $\{\torsionu(C_*^A)\mid [A]\in\Kwr(\ZZ G)\}$ generates $\Kwr(\ZZ G)$ as a group. Since $f$ agrees with $\det$ on this generating set, we conclude that $f=\det$.
\end{proof}

\section{Agrarian Polytope}
Our aim for this section is to extract from the agrarian torsion of a $G$-CW-complex a finite combinatorial object, namely a polytope. The polytope arises from a Newton polytope construction applied to elements of the Ore localisation of a twisted polynomial ring over the target skew field of an agrarian map. Our motivation for studying the Newton polytope of a torsion invariant originates from~\cite{FL2017}.

\subsection{The polytope group} Before we get to define the Newton polytope of an element of a twisted group ring, we have to introduce several concepts related to polytopes.

\begin{defi}
Let $V$ be a finite-dimensional real vector space. A \emph{polytope} in $V$ is the convex hull of finitely many points in $V$.
For a polytope $P\subset V$ and a linear map $\phi\colon V\to\RR$ we define
\[F_\phi(P)\coloneqq \{p\in P \mid \phi(p)=\min_{q\in P} \phi(q)\}\]
and call this polytope the \emph{$\phi$-face} of $P$. The elements of the collection
\[\{F_\phi(P)\mid \phi\colon V\to\RR\}\]
 are the \emph{faces} of $P$. A face is called a \emph{vertex} if it consists of a single point.
\end{defi}

Note that with the above definition polytopes are always compact and convex.

In the following, we will always take $V=\RR \otimes_{\ZZ} H$ for some finitely generated free abelian group $H$.

\begin{defi}
A polytope $P$ in $V$ is called \emph{integral} if its vertices lie on the lattice $H\subset V$.
\end{defi}

Given two (integral) polytopes $P$ and $Q$ in $V$ we can form the pointwise sum $P + Q=\{p+q\mid p\in P, q\in Q\}$. As a subset of $V$ this will again be an (integral) polytope, the vertices of which are obtained as pointwise sums of some of the vertices of $P$ and $Q$. The polytope $P+Q$ is called the \emph{Minkowski sum} of $P$ and $Q$. With the Minkowski sum the set of all (integral) polytopes in $V$ becomes an abelian monoid with neutral element $\{0\}$. It is cancellative, i.e., $P+Q=P'+Q$ for polytopes $P, P'$ and $Q$ implies $P=P'$, see~\cite[Lemma~2]{Radstrom1952}. The monoid can thus be embedded into an abelian group in a universal way and gives rise to the following algebraic object introduced in~\cite[6.3]{FT2015}:
\begin{defi}
Let $H$ be a finitely generated free abelian group. Denote by $\Poly(H)$ the \emph{polytope group} of $H$, that is the Grothendieck group of the cancellative abelian monoid given by all integral polytopes in $\RR \otimes_{\ZZ} H$ under Minkowski sum. In other words, let $\Poly(H)$ be the abelian group with generators the formal differences $P-Q$ of integral polytopes and relations $(P-Q)+(P'-Q')=(P+P')-(Q-Q')$ as well as $P-Q=P'-Q'$ if $P+Q'=P'+ Q'$. The neutral element is given by the one-point polytope $\{0\}$, which we will drop from the notation. We view $H$ as a subgroup of $\Poly(H)$ via the map $h\mapsto \{h\}$. The \emph{translation-invariant polytope group} of $H$, denoted by $\PolyT(H)$, is defined to be the quotient group $\Poly(H)/H$.
\end{defi}

While a general element of the polytope group can have many equivalent representations as a formal difference of polytopes, an element of the type $P-0$, or $P$ for short, is uniquely presented in this form. Such an element is called a \emph{single polytope}.

\subsection{The polytope homomorphism} As is the case for the $L^2$-torsion polytope, the following simple construction underpins the definition of the agrarian polytope:
\begin{defi}
Let $D$ be a skew field and let $H$ be a finitely generated free abelian group. Let $DH$ denote some twisted group ring formed out of $D$ and $H$. The \emph{Newton polytope} $P(p)$ of an element $p=\sum_{h\in H} u_h \ast h\in DH$ is the convex hull of the \emph{support} $\supp(p)=\{h\in H\mid u_h\neq 0\}$ in $H_1(H;\RR)$.
\end{defi}

Since $H$ is finitely generated free abelian, as in \cref{defi:agrarian:ore_field} we can consider the Ore field of fractions $\Ore(DH)$ of the twisted group ring $DH$. The previous definition can then be extended to elements of $\Ore(DH)$ as follows:
\begin{defi}
The group homomorphism
\begin{align*}
    P\colon \Dab[\Ore(DH)] &\to \Poly(H)\\
    pq^{-1}&\mapsto P(p)-P(q)
\end{align*}
is called the \emph{polytope homomorphism} of $\Ore(DH)$. It induces a homomorphism
\[P\colon \Dab[\Ore(DH)]/\{\pm h\mid h\in H\} \to \PolyT(H).\]
\end{defi}
The well-definedness of $P$ is immediate from the construction of $\Poly(H)$.
The fact that $P$ is a group homomorphisms is not hard, and has been shown in~\cite[Lemma 3.12]{Kielak2018} (see also the discussion following the lemma).

\subsection{The agrarian polytope for rational agrarian maps} We now consider a finitely generated group $G$ and take the free abelian group $H$ to be the free part of the abelianisation of $G$. Furthermore, we denote the canonical projection onto $H$ by $\pr$ and its kernel by $K$.

In~\cite{FL2017}, the Newton polytope is constructed for the Linnell skew field $\linnelld(G)$, which can be expressed as an Ore localisation of the twisted group ring $\linnelld(K)H$. While the target of an arbitrary agrarian map for $G$ is of course not always an Ore localisation of a suitable twisted group ring, this will be the case for the \emph{rational} agrarian maps introduced in \cref{defi:agrarian:rational}.

\begin{defi}
\label{defi:polytope:agrarian:chain_complex}
Let $\alpha\colon \ZZ G\to D$ be a rational agrarian map for $G$ and consider a $D$-acyclic finite based free $\ZZ G$-chain complex $C_*$.
The \emph{$D$-agrarian polytope} of $C_*$ is defined as
\[\Pag{D}(C_*) \coloneqq P(-\torsionag[D](C_*)) \in \Poly(H).\]
\end{defi}

The purpose of the sign in the definition of the $D$-agrarian polytope is to get a single polytope in many cases of interest. A priori, the polytope may depend on the choice of the section of $\pr$ involved in the construction of $D$. By \cref{lemm:agrarian:equivariant_embedding}, the twisted group rings obtained from any two choices differ by an isomorphism preserving supports, and thus the agrarian polytope is indeed well-defined.

We will mostly be interested in the agrarian polytope associated to the cellular chain complex of a $G$-CW-complex, where we have to account for the indeterminacy caused by the need to choose a basis of cells:
\begin{defi}
\label{defi:polytope:agrarian}
Let $\alpha\colon \ZZ G\to D$ be a rational agrarian map for $G$ and consider a $D$-acylic finite free $G$-CW-complex $X$.
The \emph{$D$-agrarian polytope} of $X$ is defined as
\[\Pag{D}(X) \coloneqq P(-\torsionag[D](X)) \in \PolyT(H).\]
\end{defi}

\begin{prop}
\label{prop:polytope:invariance}
The $D$-agrarian polytope $\Pag{D}(X)$ is a $G$-homotopy invariant of $X$.
\end{prop}
\begin{proof}
Let $X$ and $X'$ be $D$-acyclic finite free $G$-CW-complexes $G$-homotopy equivalent via $f\colon X\to X'$. We denote the induced homotopy equivalence between $X/G$ and $X'/G$ by $\overbar f$. By \cref{lemm:torsion:invariance}, the agrarian torsions of $X$ and $X'$ are related via
\[\torsionag[D](X')-\torsionag[D](X)=\det\nolimits_D(\rho(\overbar f)).\]
After applying the polytope homomorphism, we obtain
\[\Pag{D}(X')-\Pag{D}(X)=P(\det\nolimits_D(\rho(\overbar f))).\]
The latter polytope is a singleton by~\cite[Corollary~5.16]{Kielak2018} and hence $\Pag{D}(X')=\Pag{D}(X)\in\PolyT(H)$.
\end{proof}

Because of the previous proposition, the agrarian polytope of the universal covering of the classifying space of a group, which is only well-defined up to $G$-homotopy equivalence, does not depend on the choice of a particular $G$-CW-model. We are thus led to
\begin{defi}
\label{defi:polytope:agrarian:group}
Assume that $G$ is \emph{of type $\mathtt{F}$}, i.e., let there be an unequivariantly contractible finite free $G$-CW-complex $EG$.
Let $\alpha\colon\ZZ G\to D$ be a rational agrarian map for $G$.
We say that $G$ is \emph{$D$-acyclic} if any such $G$-CW complex is $D$-acyclic. If this is the case, we define the \emph{$D$-agrarian polytope of $G$} to be
\[\Pag{D}(G)\coloneqq \Pag{D}(EG).\]
\end{defi}

For future reference, we record the following direct consequence of \cref{lemm:torsion:additivity}:
\begin{lemm}
Let $0\to C'_*\to C_*\to C''_*\to 0$ be a short exact sequence of finite based free $\ZZ G$-chain complexes such that the preferred basis of $C_*$ is composed of the preferred basis of $C'_*$ and preimages of the preferred basis elements of $C''_*$. Assume that any two of the complexes are $D$-acyclic. Then so is the third and
\[\Pag{D}(C_*)=\Pag{D}(C'_*)+\Pag{D}(C''_*).\]
\end{lemm}

\subsection{The agrarian polytope for arbitrary agrarian maps}
Let $\alpha\colon\ZZ G\to D$ be an agrarian map for a finitely generated group $G$, and denote the free part of the abelianisation of $G$ by $H$.
In this situation, we can pass to the rationalisation $\ZZ G\to \Ore(DH)$ of $\alpha$ as introduced in \cref{defi:agrarian:rational}, which is always rational.
Via this replacement, we can extend \cref{defi:polytope:agrarian:chain_complex,defi:polytope:agrarian,defi:polytope:agrarian:group} to arbitrary agrarian maps.
Two remarks are in order.

First, note that it is not clear that a chain complex $C_*$ that is $D$-acyclic is also $\Ore(DH)$-acyclic (nor vice versa).
Hence, in order to compute an agrarian polytope with respect to an arbitrary agrarian map, it is always necessary to check acyclicity with respect to its rationalisation.

Second, we now have introduced two potentially different definitions of the agrarian polytope for an agrarian map $\ZZ G\to D$ that is already rational: We could calculate the polytope directly with respect to this map or first replace it by its rationalisation.
As it turns out, these two a priori different approaches lead to the same agrarian polytope.
By verifying that our two definitions are compatible, we will as a byproduct establish a comparison with the $L^2$-torsion polytope.

We will first show that passing to the rationalisation of an agrarian map $\alpha\colon\ZZ G\to D$ that is already rational only changes the agrarian torsion by pushing forward along an inclusion of skew fields:
\begin{lemm}
\label{lemm:polytope:rational_torsion}
Let $G$ be a finitely generated group and $\alpha\colon\ZZ G\to D$ a rational agrarian map.
Denote the rationalisation of $\alpha$ by $\alpha_r$ and its target skew field by $D_r$.
Then $\alpha_r$ factors through $\alpha$, and hence any finite free $G$-CW-complex is $D$-acylic if and only if it is $D_r$-acyclic.
If this is the case, then
\[\torsionag[D_r](X)=j_*(\torsionag[D](X))\in \Dab[D_r]/\{\pm g\mid g\in G\},\]
where $j_*\colon \Dab[D]/\{\pm g\mid g\in G\}\to \Dab[D_r]/\{\pm g\mid g\in G\}$ is induced by an injective map $j\colon D\hookrightarrow D_r$ between the respective agrarian map.
\end{lemm}
\begin{proof}
As discussed in \cref{rema:betti:epic}, the agrarian maps $\alpha$ and $\alpha_r$ yield the same agrarian Betti numbers, which proves the acyclicity statement.

We now turn to the statement on agrarian torsion.
By~\cite[Lemma~10.34~(1)]{Luck2002}, a skew field (and more generally, a von Neumann regular ring) is \emph{rationally closed} in any overring, i.e., every matrix over the skew field that becomes invertible over the overring is already invertible in the skew field.
Applied to our situation, we obtain that the invertible matrices appearing in the construction of $\torsionag[D_r]$ following \cref{defi:torsion:torsion} are already invertible over the skew subfield $D$.
Since the Dieudonné determinant is by construction natural with respect to inclusions of skew fields, the second statement holds.

In our case, one can also replace the use of the lemma by the more direct observation that we may put the matrices appearing in the construction of $\torsionag[D_r]$ into an upper-triangular form using elementary matrices over the skew field $D$ since the entries of the matrices lie in $\ZZ G$, and hence in $D$. A (square) matrix in upper-triangular form over a skew field is invertible if and only if its diagonal elements are non-zero, in particular invertibility over $D_r$ implies invertibility over $D$.
\end{proof}

\begin{theo}
\label{theo:polytope:well_defined}
Let $G$ be a finitely generated group and $\alpha\colon\ZZ G\to D$ a rational agrarian map.
Denote the rationalisation of $\alpha$ by $\alpha_r$ and its target skew field by $D_r$.
Let $X$ be a $D$- or $D_r$-acyclic finite free $G$-CW complex. Then $X$ is both $D$- and $D_r$-acyclic and
\[\Pag{D_r}(X) = \Pag{D}(X)\in\PolyT(H).\]
\end{theo}
\begin{proof}
By \cref{lemm:polytope:rational_torsion}, the agrarian torsions of $X$ with respect to $D$ and $D_r$ are related via
\begin{equation}
    \label{eq:polytope:torsion}
    \torsionag[D_r](X)=j_*(\torsionag[D](X))\in \Dab[D_r]/\{\pm g\mid g\in G\},
\end{equation}
where $j_*$ is induced by the inclusion $j\colon D\hookrightarrow D_r$.
Recall that we defined the agrarian polytope with respect to $D_r$ as $\Pag{D_r}(X)=P(-\torsionag[D_r](X))$, where we use that $D_r$ is the Ore field of fractions of the twisted group ring $DH$.
Analogously, the agrarian polytope with respect to $D$ is defined as $\Pag{D}(X)=P(-\torsionag[D](X))$, where we use that $D$ is the Ore field of fractions of a twisted group ring $D'H$ for some skew field $D'$.
In light of~\cref{eq:polytope:torsion}, we are thus left to check that taking the support over $D'H$ gives the same result as pushing forward to $DH$ using $j$ and taking supports there.
But $j$ restricts to an inclusion of twisted group rings and thus preserves supports.
\end{proof}

\subsection{\texorpdfstring{Comparison with the $L^2$-torsion polytope}{Comparison with the L\^{}2-torsion polytope}}

We will now apply the results of the previous section to the Linnell skew field $\linnelld(G)$ of a finitely generated torsion-free group $G$ that satisfies the Atiyah conjecture.
As before, we denote the free part of the abelianisation of $G$ by $H$ and the kernel of the projection of $G$ onto $H$ by $K$.
Our aim is to compare the $L^2$-torsion polytope introduced in~\cite{FL2017} to the agrarian polytope associated to the agrarian embedding $\ZZ G\hookrightarrow \linnelld(G)$.

The most important feature of the Linnell skew field $\linnelld(G)$ is that it is the Ore field of fractions of a twisted group ring in which $\ZZ K$ embeds, namely $\linnelld(K)H$, see~\cite[Lemma~10.69]{Luck2002}.
Translated into the agrarian language, this means that the agrarian embedding $G\hookrightarrow \linnelld(G)$ is rational.
We have already observed in \cref{theo:torsion:linnelld} that universal $L^2$-torsion and agrarian torsion define the same element in the abelianised units of the Linnell skew field, whenever both are defined.
The construction of the $L^2$-torsion polytope in~\cite[Definition~4.21]{FL2017} is thus equivalent to our definition of the agrarian polytope with respect to the rational agrarian embedding $\ZZ G\hookrightarrow \linnelld(G)$.
By \cref{theo:polytope:well_defined}, the resulting polytope agrees with the agrarian polytope constructed using the rationalisation of $\ZZ G\hookrightarrow \linnelld(G)$, i.e., without using the special structure of the Linnell skew field.

Summarising our discussion, we have established
\begin{theo}
Let $G$ be a finitely generated torsion-free group that satisfies the Atiyah conjecture and consider the agrarian embedding $\ZZ G\hookrightarrow \linnelld(G)$ into the Linnell skew field.
Denote by $\linnelld(G)_r$ the target of the rationalisation of this embedding.
Let $X$ be an $L^2$-acyclic finite free $G$-CW complex. Then $X$ is both $\linnelld(G)$- and $\linnelld(G)_r$-acyclic and
\[ \PLtwo(X)=\Pag{\linnelld(G)}(X) =\Pag{\linnelld(G)_r}(X)\in\PolyT(H).\]
\end{theo}

\section{Application to deficiency 1 groups}

\subsection{The Bieri--Neumann--Strebel invariants and HNN extensions} In order to discuss some application of the theory of agrarian invariants, we need to first cover the BNS invariants and the HNN extensions.

\begin{defi}
Let $G$ be a group generated by a finite subset $S$, and let $X$ denote the Cayley graph of $G$ with respect to $S$. Recall that the vertex set of $X$ coincides with $G$. We define the \emph{Bieri--Neumann--Strebel} (or \emph{BNS}) \emph{invariant} $\Sigma^1(G)$ to be the subset of $H^1(G;\RR) \smallsetminus \{0\}$ consisting of the non-trivial homomorphisms (the \emph{characters}) $\phi \colon G \to \RR$ for which the full subgraph of $X$ spanned by $\phi^{-1}([0,\infty))\subseteq G$ is connected.
\end{defi}

The BNS invariants were introduced by Bieri, Neumann and Strebel in~\cite{Bierietal1987} via a different, but equivalent definition. It is an easy exercise to see that $\Sigma^1(G)$ is independent of the choice of the finite generating set $S$.

We now aim to give an interpretation of lying in the BNS invariant for integral characters $\phi \colon G \to \ZZ$. To do so, we need to introduce the notion of HNN extensions.

\begin{defi}
\label{defi:bns:bns}
Let $A$ be a group and let $\sigma\colon B\xrightarrow{\cong} C$ be an isomorphism between two subgroups of $A$. Choose a presentation $\langle S\mid R\rangle$ of $A$ and let $t$ be a new symbol not in $S$. Then the group $A*_\sigma$ defined by the presentation
\[\langle S, t\mid R, tbt^{-1}=\sigma(b)\ \forall b\in B\rangle\]
is called the \emph{HNN extension of $A$ relative to $\sigma\colon B \xrightarrow{\cong}C$}. We call $A$ the \emph{base group} and $B$ the \emph{associated group} of the HNN extension.

The HNN extension is called \emph{ascending} if $B = A$.

The homomorphism $\phi \colon A*_\sigma \to \ZZ$ given by $\phi(t) = 1$ and $\phi(s) = 0$ for every $s \in S$ is the \emph{induced character}.
\end{defi}

\begin{prop}[{\cite[Proposition 4.3]{Bierietal1987}}]
\label{prop:bns:hnn}
Let $G$ be a finitely generated group, and let $\phi \colon G \to \ZZ$ be a non-trivial character. We have $\phi \in \Sigma^1(G)$ if and only if $G$ is isomorphic to an ascending HNN extension with finitely generated base group and induced character $\phi$.
\end{prop}

\begin{defi}
\label{marked def}
Suppose that $G$ is finitely generated.
Let $P$ be a polytope in the $\RR$-vector space $H_1(G;\RR)$, and let $F$ be a face of $P$. A \emph{dual} of $F$ is a connected component of the subspace
\[
\{ \phi \in H^1(G;\RR) \mid F_\phi(P) = F \}.
\]

A \emph{marked polytope} is a pair $(P,m)$, where $P$ is a polytope in $H_1(G;\RR)$, and $m$ is a \emph{marking}, that is a function $m \colon H^1(G;\RR) \to \{0,1\}$, which is constant on duals of faces of $F$, and such that $m^{-1}(1)$ is open.

The pair $(P,m)$ is a \emph{polytope with marked vertices} if $m^{-1}(1)$ is a union of some duals of vertices of $P$.

The marking $m$ will usually be implicit, and the characters $\phi$ with $m(\phi) = 1$ will be called \emph{marked}.
\end{defi}

\subsection{BNS invariants for agrarian deficiency 1 groups}
\label{sec defic 1}

In this section we will see how agrarian polytopes connect with the theory of BNS invariants of general deficiency 1 groups.

\begin{defi}
Let $G$ be a finitely presented group. The \emph{deficiency} of $G$ is the minimum taken over all finite presentations of $G$ of the difference between the number of generators and the number of relators.
\end{defi}

\begin{theo}[{\cite[Theorem 7.2]{Bierietal1987}}]
If $G$ is a finitely presented group of deficiency at least $2$, then $\Sigma^1(G) = \emptyset$.
\end{theo}

Because of the above result, we will focus on groups of deficiency $1$. This is a very rich family of groups, containing all ascending HNN extensions of finitely generated free groups, and almost all two-generator one-relator groups.

The structure of $\Sigma^1(G)$ for agrarian groups of deficiency 1 was studied by the second author in~\cite[Section 5.6]{Kielak2018}. There, stronger results were obtained for deficiency 1 groups when they were assumed to satisfy the Atiyah conjecture. The reason for this was that there was no theory of agrarian polytopes available, in contrast to the theory of $L^2$-torsion polytopes. Since we have developed the missing theory here, we can now strengthen the results and prove the following.

\begin{theo}
\label{theo:deficency:marked}
Let $G$ be a $D$-agrarian group of deficiency 1 with agrarian embedding $\ZZ G\hookrightarrow D$.
Let $D_r$ be the target of the rationalisation of this embedding in the sense of \cref{defi:agrarian:rational}.
Then $G$ is $D_r$-acyclic and there exists a marking of the vertices of the agrarian polytope $\Pag{D_r}(G)$ such that for every $\phi \in H^1(G;\RR) \smallsetminus \{0\}$ we have $\phi \in \Sigma^1(G)$ if and only if $\phi$ is marked.
\end{theo}

Before proceeding to the proof, let us first contrast the statement with~\cite[Theorem 5.23]{Kielak2018}: here, we prove that $\Sigma^1(G)$ coincides with the marked characters coming from a polytope \emph{with marked vertices}, whereas in~\cite{Kielak2018} the marking came from a marked polytope, and hence was (a priori) not as rigid as we now show it to be. In particular, we now show that if for some character $\phi \colon G \to \RR$, we can connect $\phi$ to $-\phi$ using a path lying in $\Sigma^1(G)$, then in fact $\Sigma^1(G) = H^1(G;\RR) \smallsetminus \{0\}$.

\begin{proof}[Sketch proof of \cref{theo:deficency:marked}]
The proof follows that of~\cite[Theorem 5.25]{Kielak2018} very closely, and therefore we will give here merely a sketch.

We take $X$ to be a classifying space for $G$ extending the presentation complex coming from a presentation realising the deficiency of $G$ (which is equal to $1$). We let $C_*$ denote the cellular chain complex of the universal covering of $X$; by choosing a cellular basis of $C_*$ we will treat $C_*$ as a chain complex of free finitely generated $\ZZ G$-modules.

If $\Sigma^1(G) = \emptyset$ then the result follows by taking the trivial marking $H^1(G;\RR) \to \{0\}$. We may thus assume that $\Sigma^1(G) \neq\emptyset$. Moreover, since $\Sigma^1(G)$ is open (see~\cite[Theorem A]{Bierietal1987}) and stable under positive homotheties (by definition), there exists a character $\phi \colon G \to \ZZ$ inside of $\Sigma^1(G)$. By \cref{prop:bns:hnn}, this means that we can write $G$ as an ascending HNN extension with a finitely generated base group. On the level of classifying spaces, this tells us that $G$ admits a classifying space which is a mapping torus of a selfmap of a space with finite $1$-skeleton. We conclude using \cref{theo:betti:torus} that $\bag[D_r]_1(G)=0$, where we use that the agrarian embedding $\ZZ G\hookrightarrow D_r$ is rational.

Knowing that $\bag[D_r]_1(G)=0$ allows us to immediately conclude that the boundary map
\[c_3 \colon C_3 \to C_2\]
is trivial, since it has to be trivial after tensoring with $D_r$ for reasons of dimension -- recall that $C_2$ is a free module of rank $1$ less than $C_1$, and $C_0$ has rank $1$. Therefore, the $2$-skeleton of $X$ is a classifying space for $G$, and hence we may assume that $X$ is $2$-dimensional.

Now we need to look at the chain complex $C_*$ more closely.
Since it is the cellular chain complex of the Cayley $2$-complex (the universal covering of the presentation complex), we have $C_*$ equal to
\[\ZZ G^{n-1} \xrightarrow{A} \ZZ G^n\xrightarrow{(s-1)_{s \in S}} \ZZ G\]
where $S$ is a generating set of $G$ of cardinality $n$. We enumerate $S = \{s_1, \dots, s_n\}$ and define
\[
U_i \coloneqq \{\phi \in H^1(G;\RR) \mid \phi(s_i) \neq 0 \}.
\]
We let $A_i$ denote the matrix $A$ with the row corresponding to $s_i$ removed.
Now we are exactly in the situation discussed in the proof of~\cite[Theorem 5.25]{Kielak2018} (and in~\cite[Theorem 5.23]{Kielak2018}). The argument given there shows that for every $i \in \{1, \dots, n\}$ there exists a marking $m_i$ of vertices of $P_i = P(\det D_r\otimes A_i)$ such that a character $\phi \in U_i$ lies in $\Sigma^1(G)$ if and only if $m_i(\phi) = 1$.

By mapping $C_*$ to the chain complex $\ZZ G\xrightarrow{s_i-1}\ZZ G$ concentrated in degrees 1 and 0 and considering the resulting short exact sequence of chain complexes, we see from \cref{lemm:torsion:additivity} that for every $i$ we have
\[
\Pag{D_r}(G) = P(\det D_r\otimes A_i) - P(s_i-1)
\]
Now~\cite[Lemma 5.12]{Kielak2018} allows us to construct a marking $m$ of the vertices of $\Pag{D_r}(G)$ which agrees with $m_i$ on $U_i$ for every $i$. This finishes the proof.
\end{proof}

\bibliography{main}
\end{document}